\DeclareSymbolFont{AMSb}{U}{msb}{m}{n}
\documentclass[reqno,centertags,11pt,a4paper,noamsfonts]{amsart}
\pdfoutput=1
\usepackage[svgnames]{xcolor}
\usepackage{graphicx}
\usepackage[english]{babel}
\usepackage[babel=true,expansion=alltext,protrusion=alltext-nott,final]{microtype}
\usepackage[margin={3cm},marginparwidth={2.5cm}]{geometry}
\usepackage[utf8]{inputenc}
\usepackage{libertine}
\usepackage[libertine]{newtxmath}
\useosf
\usepackage[varqu,varl]{inconsolata}
\usepackage[T1]{fontenc}
\usepackage{textcomp}
\usepackage{amsmath}
\usepackage{amsthm}
\usepackage[inline]{enumitem}
\numberwithin{equation}{section}
\usepackage[normalem]{ulem}
\usepackage{tikz}
\usetikzlibrary{arrows}
\usepackage{bm}
\usepackage{soul}
\usepackage{float}

\providecommand{\mr}[1]{\href{http://www.ams.org/mathscinet-getitem?mr=#1}{MR~#1}}
\providecommand{\zbl}[1]{\href{https://zbmath.org/?q=an:#1}{Zbl~#1}}
\providecommand{\arxiv}[1]{\href{https://arxiv.org/abs/#1}{arXiv:#1}}
\providecommand{\doi}[2]{\href{https://doi.org/#1}{#2}}
\newcommand{\RR}{\mathbb{R}}
\newcommand{\RRd}{\mathbb{R}^d}

\newcommand{\NN}{\mathbb{N}}

\newcommand{\HH}{\mathcal{H}}

\newcommand{\LL}{\mathcal{L}}

\newcommand{\1}{\mathbf{1}}

\newcommand{\dd}{\mathrm{d}}

\newcommand{\calP}{{\mathcal P}}

\newcommand{\eps}{\varepsilon}
\newcommand{\xdots}{x_{1},\dots,x_{N}}
\newcommand{\ydots}{y_{1},\dots,y_{N}}

\newcommand{\ytildedots}{\tilde{y}_{1},\dots,\tilde{y}_{N}}
\newcommand{\Prob}{\mathcal{P}}
\newcommand{\xh}{\overline{x}_h}
\newcommand{\xp}{\overline{x}_p}
\newcommand{\xpS}{ \overline{x}^{S^\mathsf{c}}_{p}}
\newcommand{\xvec}{\mathbf{x}}
\newcommand{\yvec}{\mathbf{y}}
\newcommand{\ytildevec}{\tilde{\mathbf{y}}}

\newcommand{\xvecS}{\xvec_{S^\mathsf{c}}}
\newcommand{\yvecS}{\yvec_{S^\mathsf{c}}}
\newcommand{\ytildevecS}{\ytildevec_{S^\mathsf{c}}}
\newcommand{\xbar}{\overline{x}}
\newcommand{\nup}{\nu_p}

\newcommand{\gammabar}{\overline{\gamma}}
\newcommand{\gammap}{\gamma_p}
\newcommand{\hp}{h_p}
\DeclareMathOperator{\weaklystar}{\rightharpoonup\kern-2.2ex ^* \, \,}
\DeclareMathOperator{\sgn}{sign}
\DeclareMathOperator*{\argmin}{argmin}

\DeclareMathOperator{\spt}{spt}
\DeclareMathOperator{\diam}{diam}
\DeclareMathOperator{\id}{Id}

\def\XXint#1#2#3{{\setbox0=\hbox{$#1{#2#3}{\int}$ }
\vcenter{\hbox{$#2#3$ }}\kern-.6\wd0}}

\definecolor{darkred}{rgb}{0.4,0,0} 
\definecolor{darkgreen}{rgb}{0,.4,0}

\theoremstyle{plain}
\newtheorem{theorem}{Theorem}[section]
\newtheorem{proposition}[theorem]{Proposition}
\newtheorem{corollary}[theorem]{Corollary}
\newtheorem{lemma}[theorem]{Lemma}

\newtheorem*{theorem*}{Theorem}

\theoremstyle{definition}
\newtheorem{definition}[theorem]{Definition}
\newtheorem{remark}[theorem]{Remark}

\usepackage[pagebackref=false,debug=false]{hyperref}
\hypersetup{
	linktoc=all,
	colorlinks=true,
	linkcolor=Green,citecolor=Orange,urlcolor=DarkBlue,
	pdftitle={p-Wasserstein barycenters},
	pdfauthor={Camilla Brizzi, Gero Friesecke, Tobias Ried}, 
	pdfsubject={Wasserstein barycenter, optimal transport,  multi-marginal optimal transport},
	pdfkeywords={Wasserstein barycenter, optimal transport,  multi-marginal optimal transport}} 
\begin{document}
\title{$p$-Wasserstein barycenters}
\author[C.~Brizzi]{Camilla Brizzi}
\address[C.~Brizzi]{Technische Universität München, Department of Mathematics, Boltzmannstraße 3, 85748 Garching, Germany}
\email{camilla.brizzi@tum.de}

\author[G.~Friesecke]{Gero Friesecke}
\address[G.~Friesecke]{Technische Universität München, Department of Mathematics, Boltzmannstraße 3, 85748 Garching, Germany}
\email{friesecke@tum.de}

\author[T.~Ried]{Tobias Ried}
\address[T.~Ried]{Max-Planck-Institut für Mathematik in den Naturwissenschaften, Inselstraße 22, 04103 Leipzig, Germany \\
Current address: School of Mathematics, Georgia Institute of Technology, Atlanta, GA 30332, United States of America}
\email{tobias.ried@gatech.edu}
\date{\today}
\thanks{\emph{Funding information}: Deutsche Forschungsgemeinschaft -- Project-ID 195170736 -- TRR109.\\[1ex]
\textcopyright 2024 by the authors. Faithful reproduction of this article, in its entirety, by any means is permitted for noncommercial purposes.}
\begin{abstract}
We study  barycenters of $N$ probability measures on $\RR^d$ with respect to the $p$-Wasserstein metric ($1<p<\infty$). 
We prove that \\
    -- $p$-Wasserstein barycenters of absolutely continuous measures are unique, and again absolutely\\
    \textcolor{white}{--} continuous \\
    -- $p$-Wasserstein barycenters admit a multi-marginal formulation \\
    -- the optimal multi-marginal plan is unique and of Monge form if the marginals are absolutely \\
    \textcolor{white}{--} continuous, and its support has an explicit  parametrization as a graph over any marginal space. \\
This extends the Agueh--Carlier theory of Wasserstein barycenters \cite{AC11} to exponents $p\neq 2$.
A key ingredient is a quantitative injectivity estimate for the (highly non-injective) map from $N$-point configurations to their $p$-barycenter on the support of an optimal multi-marginal plan. We also discuss the statistical meaning of $p$-Wasserstein barycenters in one dimension.
\end{abstract}
\maketitle
\tableofcontents
\section{Introduction}
Wasserstein barycenters are an important generalization of the classical notion of barycenters of points in $\RR^d$ or on a Riemannian manifold to the space of probability measures. Their applications include the interpolation of probability measures (by varying the weights) and the computation of a representative summary of input datasets (by using equal weights). Wasserstein barycenters are therefore a useful tool in data science, statistics, and image processing. 

In this article we study the properties of $p$-Wasserstein barycenters, the natural generalization of $2$-Wasserstein barycenters introduced by Agueh and Carlier \cite{AC11} to the $p$-Wasserstein distance for $1<p<\infty$. These are weighted averages of probability measures on $\RR^d$ with respect to the Wasserstein $W_p$ metric, defined for two probability measures $\mu$ and $\rho$ via
 \begin{equation*}
     W_p(\mu,\rho) \coloneq \min_{\eta\in\Pi(\mu,\rho)}\left(\int_{\RR^{2d}}|x-y|^p \, \dd\eta(x,y)\right)^{\frac1p},
 \end{equation*}
where $\Pi(\mu,\rho)$ denotes the set of transport plans from $\mu$ to $\rho$, i.e.\
\[\Pi(\mu, \rho) \coloneq \left\{ \eta \in \Prob(\RR^{2d}): \pi^1_{\#}\eta = \mu, \pi^2_{\#}\eta = \rho\right\}.\]

There are two ways to define $p$-Wasserstein barycenters. 
Let $$\mu_1, \dots, \mu_N \in \mathcal{P}_p(\RR^d)= \left\{\mu\in\mathcal{P}(\RR^d): \int_{\RR^d} |x|^p \,\dd\mu(x) < \infty\right\}$$ be probability measures on $\RR^d$ with finite $p$\textsuperscript{th} moments. 
\begin{enumerate}
	\item Most natural from the point of view of the metric structure that the Wasserstein distance $W_p$ puts on $\mathcal{P}_p(\RR^d)$, is to define 
		\begin{align}\label{eq:pW-barycenter}
			\mathrm{bar}_p((\mu_i,\lambda_i)_{i=1, \dots, N}) \coloneq	 \argmin_{\nu \in \mathcal{P}_p(\RR^d)} \sum_{i=1}^N \lambda_i W_p^p(\mu_i,\nu),
		\end{align}
		for weights $\lambda_i>0$ such that $\sum_{i=1}^N \lambda_i = 1$. 
	This definition can be seen as a generalization of the Fréchet mean of points in a geodesic space, see also \cite{LL17}. 
	\item Another way of defining the $p$-Wasserstein barycenter of the probability measures $(\mu_i)_{i=1,\dots, N}$ is by looking at a multi-marginal optimal transport problem with the measures $\mu_i$ as marginal constraints, and cost function $c_p: \RR^d \times \dots \times \RR^d \to [0,\infty)$ given by 
		\begin{align}\label{eq:p-cost}
			c_p(x_1, \dots, x_N) \coloneq \min_{z\in\RR^d} \sum_{i=1}^N \lambda_i |x_i-z|^p = \sum_{i=1}^N \lambda_i |x_i - \xbar_p(x_1, \dots, x_N)|^p,
		\end{align}
		where the classical \emph{$p$-barycenter} 
		\begin{align}\label{eq:p-barycenter}
			\xbar_p(x_1, \dots, x_N) \coloneq \argmin_{z\in\RR^d} \sum_{i=1}^N \lambda_i |x_i-z|^p
		\end{align}
  of $\xvec=(x_1, \dots, x_N)$ is the unique\footnote{Existence and uniqueness of this point follows from strict convexity and coercivity of the function $w\mapsto|w|^p$.} point where the minimum in \eqref{eq:p-cost} is attained. 
		Then consider the multi-marginal optimal transport (MMOT) problem 
		\begin{align}\label{MMpbary}\tag{MM-$p$-bar}
			C_{p\text{-MM}} \coloneq \min_{\gamma \in \Pi(\mu_1, \dots, \mu_N)} \int_{\RR^{Nd}} c_p(x_1, \dots, x_N) \,\dd\gamma(\xdots),
		\end{align}
		where\footnote{We denote by $\pi^i$ the projection on the $i^{\text{th}}$ marginal space, $\pi^i:\RR^{Nd}\to\RR$, $\pi^i(\xvec)=x_i$.}  
		\begin{equation*}
			\Pi(\mu_1,\dots,\mu_N):=\{\gamma\in\Prob(\RR^{Nd}) \, : \, \pi^i_\sharp \gamma = \mu_i, \ \text{for all} \ i=1,\dots,N \}
		\end{equation*}
		is the set of admissible transport plans between the marginals $\mu_1, \dots, \mu_N$.	
    Note that $C_{p\text{-MM}}<+\infty$ if $\mu_1, \dots,\mu_N \in \mathcal{P}_p(\RR^d)$.\footnote{\label{foot:finitecost}Indeed, since $\min_{z\in\RR^d} \sum_{i=1}^N\lambda_i |x_i-z|^p \leq  \sum_{i=1}^N\lambda_i |x_i|^p$ for all $\xvec \in \RR^{Nd}$, it follows that 
    \begin{align*}
        \int_{\RR^{Nd}} c_p(\xvec)\,\dd\gamma(\xvec) 
        \leq \sum_{i=1}^N \lambda_i \int_{\RR^{Nd}} |x_i|^p\,\dd\gamma(\xvec)
        = \sum_{i=1}^N \lambda_i \int_{\RR^{d}} |x_i|^p\,\dd\mu_i(x_i) 
        <+\infty.
    \end{align*}
    }
\end{enumerate}
Standard results in optimal transport theory, see for instance \cite{Vil09,San15}, which apply also to the multi-marginal setting, guarantee the existence of an optimizer $\gamma_p$. For a recent overview of the multi-marginal theory we refer to \cite{F24}, in particular Theorems~3.1 and~3.4 therein regarding  existence and  Kantorovich duality respectively.

Any optimizer $\gamma_p$ has the property that $(\xbar_p)_{\#}\gamma_p$ is a minimizer of \eqref{eq:pW-barycenter}.
In fact, define 
		\begin{align}\label{C2Mpbary}\tag{C2M-$p$-bar}
		    C_{p\text{-C2M}} \coloneq  \min_{\nu \in \mathcal{P}_p(\RR^d)} \sum_{i=1}^N \lambda_i W_p^p(\mu_i,\nu)
		\end{align}
then the two problems are equivalent in the following sense: 
\begin{proposition}\label{prop:equivalence}
    For any $\mu_1,\dots,\mu_N\in\Prob_p(\RR^d)$ we have that 
	\begin{equation*}
		C_{p\text{-MM}}=C_{p\text{-C2M}}.
	\end{equation*}
	Moreover, $\nup$ is a minimizer for the problem \eqref{C2Mpbary} if and only if $\nup={\xp}_\sharp\gammap$, for some minimizer $\gammap$ of the problem \eqref{MMpbary}.
\end{proposition}

In particular, the above equivalence result ensures existence for the problem \eqref{C2Mpbary}. Alternatively, existence for such a problem, without relying on MMOT, can be found in \cite{CE10}. If the measures $\mu_1, \dots, \mu_N$ are compactly supported in $\RR^d$, existence can also be obtained from the Hopf--Rinow--Cohn--Vossen Theorem (see, e.g., \cite[Theorem 2.5.28]{BBI01}) based on the fact that the geodesic space $(\mathcal{P}(X), W_p)$ is \emph{locally compact} if $X\subset \RR^d$ compact, see \cite{LL17} for details.
For $N=2$, by varying the weight in the $p$-Wasserstein barycenter of two probability measures, the resulting curve in the space of probability measures  coincides with the $p$-Wasserstein geodesic (displacement interpolation) between the two measures.  

The notion of barycenters of measures can also be extended to allow for measures with unequal masses by replacing the $p$-Wasserstein distance by the \emph{Hellinger-Kantorovich (HK) distance} in \eqref{C2Mpbary}. The resulting HK-barycenters can equivalently be characterized through a multi-marginal optimal transport problem, see \cite{CP21,FMS21} for further details. 

An advantage of the formulation as MMOT problem is that it is a linear programming problem, albeit a high-dimensional one. 
An important question therefore concerns the dimension and structure of the support of the optimal transport plan. Results in this direction have been obtained in the following situations: 
\begin{enumerate}
	\item Gangbo and \'Swie\k{}ch \cite{GS98} proved sparsity of the support of the optimal coupling $\gamma_2$ for a pairwise quadratic cost, which is equivalent to \eqref{MMpbary} for $p=2$, if all the marginals $\mu_1, \dots, \mu_N$ are absolutely continuous with respect to the Lebesgue measure on $\RR^d$. More precisely, they show that $\gamma_2$ has Monge structure, i.e.\ there exist functions $T_i:X_1 \to X_i$, $i=2, \dots, N$, such that
	\begin{equation}\label{eq:GSform}
		\gamma_2 = (\id, T_2, \dots, T_N)_{\#}\mu_1,
	\end{equation}
	which amounts to a dramatic decrease of dimensionality. The connection to $2$-Wasserstein barycenters was later made by Agueh and Carlier in \cite{AC11}. In view of the $2$-marginal optimal transport theory, the Gangbo--\'Swie\k{}ch result can be viewed as a multi-marginal version of Brenier's theorem \cite{Bre91}. 
    
    Their sparsity result has been slightly extended by Heinich \cite{Hei02} (to cost function of the form $c(\xvec) = H(\sum_{i=1}^N x_i)$ with $H:\RR^d \to \RR$ strictly concave) and Carlier \cite{Car03} ($d=1$, $H$ strictly $2$-monotone). For an extension to a large class of smooth cost functions which satisfy strong conditions on the second derivatives of the cost function see Pass \cite{Pas11}. 

     \item In our companion article \cite{BFR24-h} we prove sparsity of minimizers of the MMOT problem associated to $h$-Wasserstein barycenters for interactions $h:\RR^d \to [0,\infty)$ that are strictly convex $\mathcal{C}^2$ functions with non-degenerate Hessian. 
    The same result also follows from a more abstract result by Pass \cite{Pas14}, whose approach relies on deep results regarding the local dimension of the support of the optimal plan \cite{Pas12}. More precisely, he shows that a certain bilinear form depending on some (but not all) of the second derivatives of the cost function $c$ in a point of the support of the optimal plan can be used to derive upper bounds on the dimension of a Lipschitz manifold locally containing the support of the optimal transport plan. 
    
    \item Kim and Pass \cite{KP14} introduced the notion of $c$-splitting sets\footnote{Since our results do not depend on this concept, we refer the reader to the original article for details. Note that even though the condition of twistedness on $c$-splitting sets gives a general condition on the cost, it is often not easy to verify.} and proved that if the cost function $c$ is semiconcave and twisted on splitting sets, then the optimal plan is of Monge form provided that one of the marginals is absolutely continuous. 

\end{enumerate}

Here, we prove Monge structure and uniqueness of the optimal plan $\gammap$ for \eqref{MMpbary}. 
Our results in \cite{BFR24-h}, together with the current article on $p$-Wasserstein barycenters, therefore provide, at the same time, a \emph{multi-marginal version of the Gangbo-McCann theorem} \cite{GMC96} and a \emph{general $1<p<\infty$ version of the Agueh-Carlier theory} of Wasserstein barycenters \cite{AC11}. Our approach is different from that in \cite{AC11}, being based on the multi-marginal rather than the coupled two-marginal formulation. 

\begin{theorem}\label{th:sparsityplan} 
Let $\mu_1, \dots, \mu_N \in \mathcal{P}_p(\RRd)$. Then the following holds: 
\begin{enumerate}[label=\textbf{\emph{Part \Alph*}},leftmargin=0pt,labelsep=*,itemindent=*,itemsep=10pt]
\item \label{th:sparsity-A} For any $1<p<\infty$, if  $\mu_1, \dots, \mu_N \ll \LL^d$,
then there exists a unique optimal plan $\gammap$ for the problem \eqref{MMpbary}, and there exist measurable maps $T_i:\mathrm{supp}\mu_1\to \mathrm{supp}\mu_i$, $i=2, \dots, N$, such that
\begin{equation}\label{eq:monge-degenerate}
    \gammap=(\id,T_2,\dots,T_N)_\sharp\mu_1.
\end{equation}
Moreover, if $p\ge 2$, the same result holds with the weaker assumption that $\mu_1\ll \LL^d$.\footnote{Clearly, one can choose any other marginal to be absolutely continuous with respect to Lebesgue measure on $\RR^d$. In this case $\spt\gammap$ is parameterized as a graph over the support of that marginal.}
\item  \label{th:sparsity-B} For any $p>1$, if $\mu_1,\dots,\mu_N \ll \LL^d$, then there exist a.e.\ differentiable functions $\varphi_i: \mathrm{supp}\mu_i \to\RR$ such that
\begin{equation}\label{eq:mapT-degenerate}
    T_i=S_i^{-1} \circ S_1, \quad \text{with} \quad S_i \coloneq \id - (p\lambda_i)^{-\frac{1}{p-1}} |D\varphi_i|^{q-2} D\varphi_i, 
\end{equation} for any $i=2,\dots,N$, where $\frac 1p + \frac 1q=1$. 
\end{enumerate}
\end{theorem}
\begin{remark}\label{rem: coupled two marg pot}
The functions $\varphi_1,\dots,\varphi_N$ are Kantorovich potentials, i.e.\ optimizers of the dual problem (see Theorem~\ref{th:duality}) associated to \eqref{MMpbary}, whose existence is proved in Theorem~\ref{th:duality}. Moreover, as proved in Corollary~\ref{cor: coupled two marg pot}, each $\varphi_i$, together with its $\lambda_i|\cdot|^p$-conjugate $\varphi^{p,\lambda_i}$, is a Kantorovich potential for $\lambda_i$ times the $p$-Wasserstein distance between the measure $\mu_i$ and the $p$-Wasserstein barycenter $\nup=\mathrm{bar}_p((\mu_i,\lambda_i)_{i=1, \dots, N})$, i.e. for $\lambda_iW_p(\mu_i,\nup)$. 
Therefore the functions $S_i$ are nothing else than the optimal maps for that problem.

The existence of optimizers $\varphi_i \in L^1_{\mu_i}$ has been obtained by Kellerer \cite{K84} for a large class of cost functions, including the case of $p$-costs considered here. We give a more direct proof tailored to the $p$-case (inspired by \cite{AC11}), which automatically yields the continuity and almost-everywhere differentiability of the optimizers. 

Other examples for the existence of optimal potentials in a multi-marginal optimal transport problem with a cost function not included in the class considered by Kellerer are, for instance, the pairwise Coulomb cost with equal marginals in \cite{DP15}.
\end{remark}

Our strategy for the proof of Theorem~\ref{th:sparsityplan} is to prove the absolute continuity of the $p$-Wasserstein barycenter $\nup$ of the measures $\mu_1, \dots, \mu_N$ if one, respectively all, the measures are absolutely continuous with respect to Lebesgue measure $\LL^d$ on $\RR^d$. In fact, absolute continuity does not require optimality of the transport plan $\gammap$ in the multi-marginal formulation, but it holds under the weaker condition of $c_p$-monotonicity of its support, see Definition~\ref{def: c-monotonicity} and Proposition~\ref{prop:optimalplanarecmonotone} for details. 

\begin{theorem}\label{th:absolutecontinuity}
Let $1<p<\infty$ and $\gamma_p \in \Pi(\mu_1, \dots, \mu_N)$ have $c_p$-monotone support. Then under the condition that $\mu_1, \dots, \mu_N \ll \LL^d$
there holds
\begin{equation*}
    (\overline{x}_{p})_{\sharp}\gamma_p\ll\LL^d.
\end{equation*}
In particular, by Propostion~\ref{prop:equivalence} it follows that the $p$-Wasserstein barycenter $\nup= \mathrm{bar}_p((\mu_i,\lambda_i)_{i=1, \dots, N})$ of the measures $\mu_1, \dots, \mu_N$ with weights $\lambda_1, \dots, \lambda_N$ is absolutely continuous with respect to Lebesgue measure on $\RR^d$.\\
Moreover, if $p\ge 2$  the same result holds with the weaker assumption that only one of the measures is absoultely continuous with respect to Lebesgue measure on $\RR^d$.
\end{theorem}

We therefore prove absolute continuity of $p$-Wasserstein barycenters for $1<p<\infty$ under natural assumptions on the marginals. In a different direction, Kim and Pass \cite{KP17} proved the absolute continuity (with respect to volume measure) of Wasserstein barycenters of probability measures on compact Riemannian manifolds (with respect to the squared Riemannian distance as underlying cost function).

In our companion paper \cite{BFR24-h} we treated the case of barycenters with respect to strictly convex functions $h\in \mathcal{C}^2$ with non-degenerate Hessian. Unfortunately, this does not cover any of the $W_p$ distances except $p=2$. The main novelty in \cite{BFR24-h} was to combine the optimality of the multi-marginal coupling $\overline{\gamma}$ (in the weaker form of $c$-monotonicity of its support) with the continuity and non-degeneracy of $D^2 h$ to obtain  an injectivity estimate for the map $(\xdots) \mapsto \argmin_{z\in\RR^d} \sum_{i=1}^N \lambda_i h(x_i - z)$ on the support of $\overline{\gamma}$, which is a priori highly non-injective. This can be seen as a local Lipschitz estimate on the inverse of that map and allowed us to prove absolute continuity of the barycenter $\overline{\nu}$. 
Loosely speaking, the idea is somewhat analogous to the equality of the Hausdorff measure and the integralgeometric measure of a rectifiable set (generalized Crofton formula), see \cite[Theorem 3.3.13]{Fed96}.

Our proof in \cite{BFR24-h}, however, very much hinges on the continuity of $D^2 h$ at all points and on its non-degeneracy in a dense set of points\footnote{Indeed, even if not explicitly stated there, a careful inspection of the proof of Lemma~5.2 in \cite{BFR24-h} shows that what is really needed is that the union of any open neighbourhood of the points where $D^2h>0$ is an open covering of $\gammap$.} in the support of the optimal coupling $\gammap$. In the article at hand we overcome this difficulty by making use of the following observations: 
\begin{enumerate}[label=(\roman*)]
    \item\label{item:observation-singular} For $p>2$ the Hessian of $z\mapsto|z|^p$ degenerates only at $z=0$ and thus, given the $p$-barycenter $\xbar_p$ of the points $x_1, \dots, x_N$, the Hessian of $|\cdot|^p$ evaluated in $x_i-\xbar_p$ equals the zero matrix if and only if $\xbar_p(\xdots)=x_i$, with $i \in \{1, \dots, N\}$. In this case, we call $\xvec = (x_1, \dots, x_N)$ a \emph{singular point}. Similarly, for $1<p<2$, the Hessian of $|\cdot|^p$ is not defined at these singular points. 
    In both situations, however, we have additional information about the barycenter.
	\item The $c$-monotonicity estimate of \cite{BFR24-h} implies a local mononicity result around all regular points (i.e., the complement of the set of singular points defined in \ref{item:observation-singular}); this allows us to define a countable cover of all the regular points of the support of a $p$-optimal coupling where the inverse Lipschitz estimate of \cite{BFR24-h} holds.  
	\item In order to get absolute continuity of the barycenter we make the following \emph{important observation}: we do not need the full inverse Lipschitz estimate à la \cite{BFR24-h}, but have to estimate the injectivity of the barycenter only in terms of those coordinates for which the respective marginals are absolutely continuous. In this weaker version, such an estimate also holds on the set of singular points. We can then cover the full support of an optimal $\gammap$ with measurable sets on which an injectivity estimate holds.
 \begin{description}
	\item[$p>2$] In this case one needs only one marginal, say $\mu_1$, to be absolutely continuous.
        Then for all the points such that the $p$-barycenter is different from $x_1$, we rewrite our problem in a lower-dimensional setting, where the Hessian does not degenerate. This allows us to get the required inverse Lipschitz estimate.
        For the other points, where the $p$-barycenter is equal to $x_1$, this already entails a control on the distance of the barycenter.
	\item[$1<p<2$] By lack of regularity of $|\cdot|^p$ in this case, we cannot reduce the problem to a lower dimensional one with non-degenerate Hessian. This is the reason why we need all the marginals to be absolutely continuous in order to use the above observation.
 \end{description}
\end{enumerate}

One of the main advantages of our approach is that we study the properties of the support of an optimal plan as a \emph{geometric set} in a high-dimensional ambient space. This naturally brings in tools from geometric measure theory, which allows us to obtain the structure theorem (Theorem~\ref{th:sparsityplan}).
The result we thereby obtain is new and we hope opens the door to the further study of properties of Wasserstein barycenters. 
Finally we remark that in certain applications, it is natural to measure distances in the $1$-Wasserstein metric. While this limit case lies beyond the scope of the present article, we believe that our current work is an important step forward. 

\addtocontents{toc}{\setcounter{tocdepth}{-10}}
\subsection*{Outline of the article}
In Section~\ref{sec:preliminaries} we first recall the notion of $c$-monotonicity for MMOT. We then state and prove the equivalence between the multi-marginal and the coupled two-marginal formulation of the $p$-Wasserstein barycenter problem and show some properties the map $\xp$. Finally, we present the key lemma on $h$-Wasserstein barycenters from \cite{BFR24-h}, which is used in the proof of our main results. 
Section~\ref{sec:p-partA} is devoted to the proof of the Monge structure of the optimal plan, i.e. \ref{th:sparsity-A} of Theorem~\ref{th:sparsityplan}. This is done by proving the absolute continuity of the $p$-Wasserstein barycenter, first in a situation which covers the full range of $1<p<\infty$ (Theorem~\ref{th:absolutecontinuitypsmall}), but requires all the marginals to be absolutely continuous. For $p\ge 2$ we refine this to the case of only one absolutely continuous marginal (Theorem~\ref{th:absolutecontinuitypbig}). 
In Section~\ref{sec:p-partB} we prove \ref{th:sparsity-B} of Theorem~\ref{th:sparsityplan}. To this end, we show the existence and a.e. differentiability of  Kantorovich potentials for the dual problem of \eqref{MMpbary} (Theorem~\ref{th:duality}).
Finally, Section~\ref{sec:one-dim} discusses the statistical interpretation of $p$-Wasserstein barycenters in one dimension.

\addtocontents{toc}{\setcounter{tocdepth}{1}}

\section{Preliminaries}\label{sec:preliminaries}
We start by defining the natural generalization of the notion of $c$-monotonicity from $2$-marginal optimal transport to the multi-marginal setting. 
\begin{definition}\label{def: c-monotonicity}
Let $c:\RR^{Nd}\to\RR$. We say that a set $\Gamma\subset \RR^{Nd}$ is $c$-monotone if for every $\xvec^1=(x^1_1,\dots,x^1_N)$, $\xvec^2=(x^2_1,\dots,x^2_N)\in\Gamma$ we have that 
\begin{equation*}
    c(x^1_1,\dots,x^1_N)+c(x^2_1,\dots,x^2_N)\le c(x^{\sigma_1(1)}_1,\dots,x^{\sigma_N(1)}_N)+c(x^{\sigma_1(2)}_1,\dots,x^{\sigma_N(2)}_N),
\end{equation*}
where $\sigma_i\in S(2)$, with $S(2)$ the set of permutations of two elements. 
\end{definition}
We recall that also the notion of $c$-cyclical monotonicity can be defined in the multi-marginal setting (see, for instance, Definition 2.2 in \cite{KP14}) and that $c$-cyclical monotonicity implies $c$-monotonicity. If $c$ is continuous, then the support of any optimal plan $\gammabar$ for the multi-marginal OT problem associated to $c$ is $c$-cyclically monotone, see \cite[Proposition~2.3]{KP14}. Therefore we have

\begin{proposition}\label{prop:optimalplanarecmonotone}
If $c:\RR^{Nd}\to\RR$ is continuous and $\gammabar$ is optimal for
\begin{equation*}
\min_{\gamma\in\Pi(\mu_1,\dots,\mu_N)}\int_{\RR^{Nd}} c(\xdots)\,\dd\gamma,
\end{equation*} then $\spt\gammabar$ is $c$-monotone.
\end{proposition}
\subsection{Equivalence of multi-marginal and coupled two-marginal formulation}
Even though Proposition~\ref{prop:equivalence} and Corollary~\ref{cor:optimalitytwomarginals} were already proved in \cite{CE10} and in \cite{BFR24-h}, we give the proof below, because we think that the simple argument provides a better insight into the relation between the multi-marginal and coupled two-marginal formulations.

\begin{proof}
 We first show that $C_{p\text{-MM}}\ge C_{p\text{-C2M}}$. Let $\gammap$ be optimal for \eqref{MMpbary} and let $\nup=(\xp)_\sharp\gammap$. For any $i\in\{1, \dots, N\}$ we define $\gamma_i:=(\pi_i,\xp)_\sharp\gammap$, where $\pi_i(\xdots)=x_i$. Then $\gamma_i\in\Pi(\mu_i,\nup)$ and
 \begin{align}
    \notag \int_{\RR^{Nd}}\lambda_i|x_i-\xp(\mathbf{x})|^p \dd\gammap(\xvec)
    &=\int_{\RR^{Nd}}\lambda_i |\pi_i(\mathbf{x})-\xp(\mathbf{x})|^p \dd\gammap(\xvec)
    =\int_{\RR^{2d}}\lambda_i|x_i-z|^p\dd\gamma_i(x_i,z)\\
     \label{eq:MMgraterC2M}&\ge \lambda_iW_p(\nup,\mu_i).
 \end{align}
 By summing over $i$ we obtain the desired inequality.

 \smallskip
 For the converse inequality $C_{p\text{-MM}}\le C_{p\text{-C2M}}$, let $\rho\in\Prob(\RRd)$ and $\gamma_i$ be optimal for $\lambda_iW_p(\rho,\mu_i)$. Denote by $\{\gamma_i^{(z)}\}_{z\in\RRd}$ the disintegration of $\gamma_i$ with respect to the first marginal, i.e.\ $\gamma_i^{(z)}\in\Prob(\RRd)$ for every $z\in\RRd$ and $\dd\gamma_i(x_i,z)=\dd\gamma_i^{(z)}(x_i)\dd\rho(z)$.
If $\gamma$ is such that  
\begin{equation}\label{eq:gammadisint}
    \dd\gamma(\xdots)=\int_{\RRd}\dd\gamma_1^{(z)}(x_1)\cdots \dd\gamma_N^{(z)}(x_N)\dd\rho(z),
\end{equation}
then
\begin{align}
\notag\sum_{i=1}^N \lambda_i W_p(\mu_i,\rho)&=\sum_{i=1}^N\lambda_i\int_{\RR^{2d}}|x_i-z|^p \dd\gamma_i
=\sum_{i=1}^N\lambda_i\int_{\RR^{2d}}|x_i-z|^p\dd\gamma_i^{(z)}(x_i)\dd\rho(z)\\\notag&=\int_{\RR^{(N+1)d}}\sum_{i=1}^N\lambda_i|x_i-z|^p\dd\gamma_1^{(z)}(x_1)\cdots \dd\gamma_N^{(z)}(x_N)\dd\rho(z)\\
\label{eq:minimalitybary}&\ge \int_{\RR^{Nd}}\sum_{i=1}^N\lambda_i|x_i-\xp(\mathbf{x})|^p \int_{\RRd}\dd\gamma_1^{(z)}(x_1)\cdots \dd\gamma_N^{(z)}(x_N)\dd\rho(z)\\
\label{eq:last}&\ge C_{p\text{-MM}}.
\end{align}
The claim follows by taking the infimum over all $\rho$.

\smallskip
We conclude by proving the equivalence of the minimizers. Let $\gammap$ be optimal for \eqref{MMpbary}. Then by equality in \eqref{eq:MMgraterC2M}, $\nup=(\xp)_\sharp \gamma$ has to be optimal for \eqref{eq:p-barycenter}. \\
Vice versa, let $\nup$ be optimal for the problem \eqref{C2Mpbary} and define $\hat\gamma\in\Pi(\mu_1,\dots,\mu_N,\nup)$ via
\begin{equation*}
\dd\hat\gamma(\xdots,z):=\dd\gamma_1^{(z)}(x_i)\cdots \dd\gamma_N^{(z)}(x_N)\dd\nup(z).
\end{equation*}
Clearly, by \eqref{eq:gammadisint}, we have that $\gamma=(\pi_1,\dots, \pi_N)_{\sharp}\hat\gamma$. Since the inequalities \eqref{eq:minimalitybary} and \eqref{eq:last} are in fact equalities, $\gamma$ has to be optimal. 
But inequality \eqref{eq:minimalitybary} is an equality if and only if $z=\xp(\mathbf{x})$ on $\spt\hat\gamma$, so that we also have that $\hat\gamma=(\xp,\id)_\sharp\gamma$ and thus $\nup=(\pi_0)_\sharp\hat\gamma=(\xp)_\sharp\gamma$.
\end{proof}
\begin{corollary}
\label{cor:optimalitytwomarginals}
Let $\gammap$ be optimal for the problem \eqref{MMpbary}. Then  $\gamma_i:=(\pi_i,\xh)_\sharp\gammap\in \Pi(\mu_i,\nup)$ is optimal for $W_p(\mu_i,\nup)$, where $\nup:=(\xp)_\sharp\gammap$ is the $p$-Wasserstein barycenter.
\end{corollary}
\begin{proof}
The fact that $\gamma_i:=(\pi_i,\xp)_\sharp\gammap\in \Pi(\mu_i,\nup)$ is straightforward from the definition. Moreover, by Proposition~\ref{prop:equivalence},
\begin{align*}
\sum_{i=1}^N \lambda_i W_p(\mu_i,\nup)&\le\sum_{i=1}^N\int_{\RR^{2d}}\lambda_i|x_i-z|^p\dd\gamma_i(x_i,z)
=\sum_{i=1}^N\int_{\RR^{Nd}}\lambda_i|\pi_i(\xvec)-\xp(\xvec)|^p \dd\gammap(\xvec)\\
&=\sum_{i=1}^N \lambda_i W_p(\mu_i,\nup).
\end{align*}
Since $\lambda_i W_p(\mu_i,\nup)\le \lambda_i\int_{\RR^{2d}}|x_i-z|^p \dd\gamma_i$ for every $i\in\{1, \dots, N\}$, the equality above yields $\lambda_i W_p(\mu_i,\nup)= \lambda_i\int_{\RR^{2d}}|x_i-z|^p \dd\gamma_i$, and thus the optimality of $\gamma_i$.
\end{proof}

\subsection{Some properties of the barycenter map $\xp$}

\begin{proposition}
\label{prop: bary in convex hull }
Let $1<p<\infty$. Then the $p$-barycenter $\xp(x_1,\dots, x_N)$ of the points $x_1, \dots, x_N \in \RR^{d}$ is contained in the convex hull of $\{x_1, \dots, x_N\}$. More precisely, we have that
\begin{align}\label{eq:convex-combination}
    \xp(x_1, \dots, x_N) = \sum_{i=1}^N \eta_i(\xvec) x_i
\end{align}
with $\eta_i(\xvec)\coloneq\frac{\lambda_i p|x_i-\xp(x)|^{p-2}}{\sum_{k=1}^{N}\lambda_kp|x_k-\xp(x)|^{p-2}}$.
\end{proposition}
\begin{proof}
Fix $\xvec=(\xdots)$, then by optimality, $\xp(\xvec)$ is the only solution of
\begin{equation}\label{eq:euler-lag bary}
   \sum_{i=1}^{N}\lambda_i p|x_i-z|^{p-2}(x_i-z)=0.
\end{equation}
Rearranging the equation yields the conclusion. 
 \end{proof}
\begin{corollary}\label{cor: continuity bary} (Continuity of the barycenter map)
    Let $1<p<\infty$. Then the functions $\xp:\RR^{Nd}\to\RRd$ and $c_p:\RR^{Nd}\to\RRd$ are continuous.
\end{corollary}
\begin{proof}
    The continuity of $(\xdots)\mapsto\xp(\xdots)$ follows from the fact that it locally is the minimum of a family of continuous functions $\{f_z\}_{z\in K}$ over a compact set $K$, where\footnote{If $B_1,\dots, B_N\subset\RRd$ are open balls, then for any $(\ydots)\in B_1\times\cdots\times B_N$, \[\xp(\ydots)=\min_{z\in K} f_z(\ydots),\]
    where $K$ is the closure of the convex hull of $\bigcup_{i=1}^N B_i$. Indeed, $K$ contains the union of the convex hulls of $\{\ydots\}$ with $(\ydots)\in B_1\times\cdots\times B_N$.} \[ f_z(\xdots)\coloneq \sum_{i=1}^N\lambda_i|x_i-z|^p. \]
    As a composition of continuous functions, also $c_p$ is continuous.
\end{proof}
\begin{corollary} (Asymptotic behavior as $p\to 1$ and $p\to\infty$)
 For any fixed $\xvec=(\xdots)\in\RR^{Nd}$, limit points of \, $\xp(\xvec)$ as $p\to 1$ or $p\to \infty$ exist. Moreover, every limit point is a solution  of 
 \begin{align*}
     \min_{z\in\RRd}\sum_{i=1}^N\lambda_i|x_i-z| \quad \text{for } p\to1, \quad \text{and} \quad \min_{z\in \RRd}\max_{i=1,\dots,N}|x_i-z| \quad \text{for }p\to\infty.
 \end{align*}
\end{corollary}
\begin{proof}
Given $\xvec\in\RR^{Nd}$, Proposition \ref{prop: bary in convex hull } ensures that $\{\xp(\xvec)\}_{1<p<\infty}$ is contained in the convex hull of $\{\xdots\}$, yielding the existence of limit points as $p\to 1$ or as $p\to\infty$. It is easy to verify that the functions 
\[F_p(z)\coloneq\left(\sum_{i=1,\dots,N}\lambda_i|x_i-z|^p\right)^{\frac 1p}\]
converge (in the sense of $\Gamma$-convergence) to
\[F_1 (z)\coloneq\sum_{i=1}^N\lambda_i|x_i-z|, \quad \text{as }p\to1,\]
and to
\[F_\infty (z):=\max_{i=1,\dots,N}|x_i-z|, \quad \text{as }p\to\infty.       \]
We conclude by observing that $\xp(\xvec)$ minimizes $F_p$ for any $1<p<\infty$. 
\end{proof}
\begin{remark}
For $d=1$ and $\lambda_1=\cdots = \lambda_N=\frac 1N$, one can prove that the functionals $F_1$ (for $N$ odd) and $F_\infty$ (for any $N$) admit a unique minimizer. It follows that the full sequence $\{\xp(\xvec)\}$ converges to $\xbar_1(\xvec)$ as $p\to 1$, and to $\xbar_\infty(\xvec)$ as $p\to\infty$, respectively. Moreover $\xbar_1(\xvec)$ is  the only (for $N$ odd) solution of the equation
\begin{equation}\label{eq:euler-lag 1}
   \sum_{i=1}^N\sgn(x_i-z)=0,
\end{equation}
and $\xbar_\infty(\xvec)$ is the only solution of 
\begin{equation}\label{eq:euler-lag infty}
    (\min_i x_i-z)+(\max_i x_i-z)=0.
\end{equation}
Hence, equations \eqref{eq:euler-lag 1} and \eqref{eq:euler-lag infty}, which  can be formally obtained by passing to the limit $p\to1 $ and $p\to\infty$ in \eqref{eq:euler-lag bary}, provide a generalization of the Euler-Lagrange equation \eqref{eq:euler-lag bary} of $F_p$ to the non smooth cases $p=1$ and $p=\infty$ .
\end{remark}
\subsection{Preliminary results for general strictly convex functions $h$}\label{subsec:preliminary results h}
Let us now consider a strictly convex function $h\in C^1(\RRd)$ with $h\ge 0$, and such that $\lim_{|z|\to+ \infty}h(z)=+\infty$. 
Given $\xvec=(\xdots)\in \RR^{Nd}$, we define \footnote{Note that $\xh$ is well-defined thanks to the strict convexity and coercivity of $h$.}
\begin{align*}\label{eq:h-barycenter}
    \xh(x_1, \dots, x_N) \coloneq \argmin_{z\in\RR^d} \sum_{i=1}^N \lambda_i h(x-z),
\end{align*}
and 
\begin{equation*}\label{eq:ch}
c_h(\xdots)\coloneq \min_{z\in\RRd}\sum_{i=1}^{N}\lambda_i h(x_i-z)=\sum_{i=1}^{N}\lambda_i h(x_i-\overline{x}_h(\mathbf{x})).
\end{equation*}
The multi-marginal optimal transport problem associated to $c_h$ is given by
\begin{equation}\label{eq:MMhbary}\tag{MM-$h$-bar}
	C_{h{\text{-MM}}} \coloneq \min_{\gamma\in\Pi(\mu_1,\dots,\mu_N)}\int_{\RR^{Nd}} c_h(\xdots)\,\dd\gamma.
\end{equation}

\begin{remark}\label{rem:hp is h}
    Notice that the function $h_p\coloneq|\cdot|^p$ satisfies the above assumptions for any $p>1$. From \eqref{eq:p-cost}, \eqref{eq:p-barycenter}, and \eqref{MMpbary}, we have that  $c_p=c_{h_p}$, $\xp = \overline{x}_{h_p}$, and $C_{p\text{-MM}} = C_{h_p\text{-MM}}$.
\end{remark}

As pointed out in the introduction, in the next section we will exploit a result of \cite{BFR24-h}, which we recall here. 
For brevity, we define, where they exist, 
\begin{align*}
     H(\xvec)\coloneq \sum_{k=1}^N \lambda_kD^2h(x_k-\xh(\xvec)), \quad \text{and} \quad
     M_i(\xvec)\coloneq D^2h(x_i-\xh(\xvec)), \quad \text{for } i\in\{1,\dots,N\}.
\end{align*}
\begin{lemma}\label{lem:firstinequality}
Assume that $\spt\overline\gamma$ is $c_h$-monotone (see Defintion \ref{def: c-monotonicity}) and that there exists a point $\xvec=(\xdots)\in\RR^{Nd}$ with the property that $h\in \mathcal{C}^2\left( \bigcap_{i=1}^N \pi^i\left(B(\xvec,\rho)\right) - B(\xh(\xvec),\delta) \right)$\footnote{here $-$ is the Minkowski difference, i.e. $A-B:=\{a-b \, : \, a\in A, \, b\in B\}$. } for some  $\rho,\delta>0$, and that $\det D^2h(x_{i_0}-\xh(\xvec))\neq 0$ for some $i_0\in\{1, \dots, N\}$.
Then for every $\eps>0$ there exists $r>0$ such that for every $\yvec,\ytildevec\in\spt\overline\gamma\cap B(\xvec,r)$,
\begin{equation}\label{eq:firstinequality}
    (y_{i_0}-\tilde{y}_{i_0})^T\left(D^2h(x_{i_0}-\xh(\xvec))\right) \left(\xh(\yvec)-\xh(\ytildevec)\right)\ge \Lambda_{i_0} |y_{i_0}-\tilde y_{i_0}|^2-\eps N(1+|M_{i_0}(\xvec)|)|\yvec-\ytildevec|^2,
\end{equation}
and $\Lambda_{i_0}>0$ is the smallest eigenvalue of the matrix  $D^2h(x_{i_0}-\xh(\xvec))H(\xvec)^{-1}D^2h(x_{i_0}-\xh(\xvec))$. 
\end{lemma}

\begin{proof} Lemma~\ref{lem:firstinequality}  is a minor improvement of \cite[Lemma~5.2]{BFR24-h} and most of the proof carries over.
In that paper the assertion was proved for $h\in \mathcal{C}^2(\RRd)$. However, a closer inspection of the proof shows the local regularity assumption $h\in \mathcal{C}^2\left( \bigcap_{i=1}^N \pi^i\left(B(\xvec,\rho)\right) -B(\xh(\xvec),\delta) \right)$ on $h$ is sufficient. 
  Indeed, the main ingredients of that proof are 
  \begin{enumerate}
      \item the $c_h$-monotonicity of  $\gammabar$ and a characterization of $c_h$ monotonicity which involves the mixed second derivatives of the cost function $c_h$;
      \item \label{a2 : strictly positive} the fact that $D^2h$ is a positive definite matrix and that $D^2h(x_{i_0}-\xh(\xvec))$ is strictly positive definite for some $i_0$;
      \item \label{a3 : localregularity} the $\mathcal{C}^1$ regularity of $\xh$ and the $\mathcal{C}^2$ regularity of $c_h$ around $\xvec$. 
  \end{enumerate}
Note that \eqref{a2 : strictly positive} is directly implied by combining the assumption of the lemma with the convexity of $h$. 
It remains to show that local regularity of $h$ and nondegenerancy of $D^2h$ at one point imply \eqref{a3 : localregularity}. This is done in the following lemma.
\end{proof}

\begin{lemma}\label{lem:diffbary}
Assume that there exists $\xvec=(\xdots)\in\RR^{Nd}$ with the property that $h\in \mathcal{C}^2\left( \bigcap_{i=1}^N \pi^i\left(B(\xvec,\rho)\right) -B(\xh(\xvec),\delta) \right)$ for some $\rho,\delta>0$, and that $\det D^2h(x_{i_0}-\xh(\xvec))\neq 0$ for some $i_0\in\{1, \dots, N\}$. Then there exists an open neighborhood $U_{\xvec}$ of $\xvec$ such that  $\xh\in \mathcal{C}^1(U_{\xvec})$ and $c_h\in\mathcal{C}^2(U_{\xvec})$.
\end{lemma}
\begin{proof}
  Due to its optimality, for any $\xvec=(\xdots)\in\RR^{Nd}$, the point $\xh(\xdots)$ is the only solution of 
 \begin{equation}\label{eq:firstorderoptimalitybary}
     \sum_{i=1}^{N}\lambda_i Dh(x_i-z)=0.
 \end{equation} 
We point out that the function $ \sum_{i=1}^{N}\lambda_i Dh(x_i-z)$ is continuously differentiable in the open set where $h$ admits a continuous second derivative and that there, since $h$ is convex, $D^2h\ge 0$ and $$\det(D^2h(x_i-\xh(\xvec)))\neq0 \implies D^2h(x_i-\xh(\xvec))>0.$$
Then $D_z\left(\sum_{i=1}^{N}\lambda_i Dh(x_i-\xh(\xvec)\right)$ is invertible, and by the Implicit Function Theorem there exists an open neighborhood $U_{\xvec}$ of $\xvec$ such that $z=\xh(\yvec)$ satisfies \eqref{eq:firstorderoptimalitybary} for every $\yvec=(\ydots)\in U_{\xvec}$ and $\xh\in \mathcal{C}^1(U_{\xvec})$. Moreover one can easily check that for every $\yvec\in U_{\xvec}$ 
$$D_{x_i}c_h(\yvec)=Dh(y_i-\xh(\yvec))$$ and thus $c_h\in\mathcal{C}^2(U_{\xvec})$.
\end{proof}

\section{Sparsity of the optimal plan}\label{sec:p-partA}
In this section, we give the proof of \ref{th:sparsity-A} of Theorem \ref{th:sparsityplan}. From now on, for any $1<p<\infty$, let $h_p:=|\cdot|^p$. As pointed out in Remark \ref{rem:hp is h}, $h_p$ satisfies all the assumptions of Section~\ref{subsec:preliminary results h}.
 \begin{remark}\label{rmk:singularities}
 For any $z\in\RR^d$, one has
 \begin{equation*}
    D\hp(z)=p|z|^{p-2}z,
 \end{equation*}
 and, whenever $D^2h_p$ exists at a point $z\in\RR^d$, there holds
 \begin{equation*}
  D^2\hp(z)=p |z|^{p-2}\left((p-2) \frac{z}{|z|}\otimes \frac{z}{|z|}+ \1 \right).
 \end{equation*}
 In particular:
 \begin{enumerate}[label=$\circ$]
     \item if $1<p<2$, $D\hp$ is not differentiable at $0$, but  $h\in \mathcal{C}^2(\RR^d\setminus \{0\})$;
     \item if $p\ge 2$, $h\in\mathcal{C}^2(\RRd)$;
     \item if $p=2$, $D^2 \hp= \mathrm{Id}$, and thus $D^2\hp>0$ everywhere;
    \item if $1<p<2$ or $p>2$, $D^2\hp>0$ on $\RR^d\setminus \{0\}$.
 \end{enumerate}
Therefore, given $\mathbf{x}:=(\xdots)\in\RR^{Nd}$ and $\xvec\mapsto \xp(\xvec)$ defined as in \eqref{eq:p-barycenter},
 for any $i=1,\dots,N$, the following holds: 
 \begin{enumerate}[label=$\circ$]
    \item for any $1<p<\infty$, 
     \begin{align}\label{eq:singularityinxiequalbarygradient}
    D\hp(x_i-\xp(\xvec))=0 \quad\text{if and only if} \quad    
     x_i=\xp(\xvec);
 \end{align}
 \item if $1<p<2$,
  \begin{equation}\label{eq:singularityinxiequalbary*}
    D^2\hp(x_i-\xp(\xvec)) \quad \text{exists if and only if} \quad  x_i\neq\xp(\xvec);
 \end{equation}
 \item for any $1< p < \infty$, $p\neq 2$
  \begin{equation}\label{eq:singularityinxiequalbary}
    \det (D^2\hp(x_i-\xh(\xvec)))\neq0 \quad\text{if and only if} \quad 
     x_i\neq\xp(\xvec),
 \end{equation}
 since  $\det (\1+u\otimes v )= 1+u\cdot v$ by the matrix determinant lemma.
\end{enumerate}
\end{remark}
Motivated by the previous remark, we now construct a partition of $\RR^{Nd}$ based on the observations \eqref{eq:singularityinxiequalbary*} and \eqref{eq:singularityinxiequalbary}. 
For every $S\subset\{1,\dots, N\}$,
we thus define 
  \begin{equation}\label{eq:defDS}
     D_S:=\{ \xvec=(\xdots)\in\RR^{Nd}: \, \xp(\xvec)=x_i \,\, \text{for every} \, i\in S, \, \text{and} \, \xp(\xvec)\neq x_i \, \text{for every} \, i\not\in S \}.
 \end{equation}
Then, clearly, $\RR^{Nd}=\underset{S}\bigcup D_S$. 

\subsection{Absolute continuity of the $W_p$-barycenter}  

\begin{proposition}\label{prop:Dzerocase}
Let $\spt\gamma_p$ be $c_p$-monotone. Then there exists a countable cover $\{U_m\}_{m\in\NN}$ of the set $D_\emptyset\cap\spt\gamma_p$ with the following property: For every $m\in\NN$ there exists $L_m>0$ such that 
 \begin{equation*}\label{eq:inverselipschitzpD0}
     |\xp(\yvec)-\xp(\ytildevec)|\ge L_m|\yvec-\ytildevec|
 \end{equation*}
 for any $\yvec=(\ydots),\ytildevec=(\ytildedots)\in D_\emptyset\cap\spt\gamma_p\cap U_m$.
 \end{proposition}

 \begin{proof}
 Let $\xvec=(\xdots)\in D_\emptyset$ and take $r=r(\xvec)$ small enough such that $B(\xvec,r)\subset D_\emptyset$ and such that $\xp(\xvec)$ does not lie in the closure of $\pi^i\left(B(\xvec,r)\right)$ for any $i\in\{1, \dots, N\}$. Such a $r$ exists because the set $D_\emptyset$ is open by continuity of the map $\xp$ (see Corollary~\ref{cor: continuity bary}); indeed, $D_\emptyset = \bigcap_{i=1}^N\left((\xp - \pi^i)^{-1}\left(\RR^{Nd}\setminus \{\mathbf{0}\}\right)\right)$. 
 Now take $\delta>0$ such that $\pi^i\left(B(\xvec,r)\right)\cap B(\xp(\xvec),\delta)=\emptyset$ for every $i=1,\dots, N$. Then $h\in \mathcal{C}^2\left( \bigcap_{i=1}^N \pi^i\left(B(\xvec,r)\right) -B(\xh(\xvec),\delta) \right) $.
 Indeed, this is clearly true for $p\ge 2$ and holds true for $1<p<2$ thanks to \eqref{eq:singularityinxiequalbary*} in Remark \ref{rmk:singularities}. Moreover, by definition of $D_{\emptyset}$ and by \eqref{eq:singularityinxiequalbary}, $\det (D^2\hp(x_i-\xp(\xvec)))\neq0$ for every $i\in\{1,\dots,N\}$. We can therefore apply Lemma \ref{lem:firstinequality}, for every $i=1,\dots,N$ and summing \eqref{eq:firstinequality} over $i$, one gets that for every $\eps>0$ there exists $r>0$, possibly smaller than the previous one, such that
\begin{equation*}
    \sum_{i=1}^N (y_i-\widetilde{y}_i)^\intercal D^2\hp(x_i-\xp(\xvec)) \left(\xp(\yvec)-\xp(\ytildevec)\right)\ge \sum_{i=1}^N \Lambda_i |y_i-\widetilde y_1|^2-\eps N\sum_{i=1}^N(1+M_i(\xvec))|\yvec-\ytildevec |^2,
\end{equation*}
for every $\yvec,\ytildevec\in B(\xvec,r)\cap\spt\gamma_p$, and therefore 
\begin{equation*}
NM(\xvec)|\yvec-\ytildevec||\xp(\yvec)-\xp(\ytildevec)|\ge \left(\Lambda_m-\eps N^2(1+M(\xvec))\right)|\yvec-\ytildevec|^2,
\end{equation*}
where $M(\xvec):=\max_i M_i(\xvec)$ and $\Lambda_m:=\min_i \Lambda_i$.
By choosing $\eps>0$ such that $\Lambda_m-\eps N^2(1+M(\xvec))>0$ and a suitable $r>0$, we get 
\begin{equation*}
    |\xp(\yvec)-\xp(\ytildevec)|\ge L(\xvec)|\yvec-\ytildevec|, 
\end{equation*}
 for every $\yvec,\ytildevec\in B(\xvec,r(\xvec))\cap \spt\gamma_p$.
Clearly $D_\emptyset\cap\spt\gamma_p\subset\bigcup_{\xvec\in D_\emptyset}B(\xvec,r)$. As every subset of $\RR^d$ is second countable, we can extract countably many points $\{\xvec_m\}$ such that $D_\emptyset\cap\spt\gamma_p\subset\bigcup_{m\in\NN}B(\xvec_m,r_m)$.
\end{proof}

\begin{lemma}\label{lem:diameterDempty}
 Let $\spt\gamma_p$ be $c_p$-monotone, and let $\{U_m\}_{m\in\NN}$ be the countable cover of $D_\emptyset\cap\spt\gamma_p$ defined in Proposition \ref{prop:Dzerocase}.

If $E\subset\RR^d$ is such that $\diam(E)<\delta$ for some $\delta>0$, then\footnote{Here and in the following, we set $\diam \emptyset= 0$, as commonly done in the context of Hausdorff measure.}
\begin{equation}\label{eq:diameterDempty}
\diam(\pi^i(\xp^{-1}(E)\cap D_\emptyset \cap U_m\cap\spt\gamma_p) )<\frac {\delta} {L_m} 
\end{equation}
for every $i=1,\dots, N$ and for every $m\in\NN$.

 In particular, if $E\subset\RRd$ is such that $\LL^d(E)=0$, then
    \begin{equation}\label{eq:hausdorffzero1}
     \LL^d(\pi^i(\xp^{-1}(E)\cap D_\emptyset \cap U_m \cap \spt\gammap))=\HH^d(\pi^i(\xp^{-1}(E)\cap D_\emptyset \cap U_m \cap \spt\gammap))=0
 \end{equation}
 for every $i=1,\dots, N$.
 \end{lemma}
 \begin{proof}
 Let us fix $U_m$ and set $\widetilde E \coloneq \xp^{-1}(E)\cap D_\emptyset\cap\spt\gammap\cap U_m$. For any $w,\widetilde w\in\widetilde E$, there exist $z,\widetilde z \in E$ such that $w\in \xp^{-1}(\{z\})\cap D_\emptyset\cap\spt\gammap\cap U_m$ and $\widetilde w\in \xp^{-1}(\{\widetilde z\})\cap D_\emptyset\cap\spt\gammap\cap U_m$. By Proposition~\ref{prop:Dzerocase}, we have that 
 \[|\pi^i(w)-\pi^i(\widetilde w)|\le\left(\sum_{i=1}^N|\pi^i(w)-\pi^i(\widetilde w)|^2\right)^{\frac 12}\le \frac{1}{L_m} |\xp(w)-\xp(\widetilde w)|= \frac{1}{L_m} |z-\widetilde z|\le \frac {\delta} {L_m}.\]
 Now, if $E\subset \RRd$ is such that $\LL^d(E)=0$, let $\{E_n\}_{n}$ be a countable cover of $E$, such that $\diam(E_n)<\delta$ for every $n$, for some $\delta>0$. Then $\{\pi^i(\xp^{-1}(E_n)\cap D_\emptyset\cap U_m \cap \spt\gammap)\}_n$ is a cover of $\pi^i(\xp^{-1}(E)\cap D_\emptyset\cap U_m \cap \spt\gammap)$ such that by inequality \eqref{eq:diameterDempty}, $\diam(\pi^i(\xp^{-1}(E_n)\cap D_\emptyset\cap U_m \cap \spt\gammap))<\frac{\delta}{L_m}$ with $L_m>0$. By the definition of Hausdorff measure this implies \eqref{eq:hausdorffzero1}. 
 
For completeness we include the details: note that 
\begin{align*}
    \HH^d_{\frac \delta L}(\pi^i(\xp^{-1}(E)\cap D_\emptyset \cap U_m \cap \spt\gammap))&\le c_d\sum\diam (\pi^i(\xp^{-1}(E)\cap D_\emptyset \cap U_m \cap \spt\gammap))^d\\&\le c_d \frac{1}{L_m^d}\sum\diam (E_n)^d.
\end{align*}
By taking the infimum over all the countable covers $\{E_n\}_{n}$  with diameter less than $\delta$, \[\HH^d_{\frac \delta L_m}(\pi^i(\xp^{-1}(E)\cap D_\emptyset\cap U_m\cap\spt\gamma_p))\le \frac{1}{L_m^d} \HH^d_{\delta}(E),\] and, passing to the limit for $\delta\to 0$, $\HH^d(\pi^i(\xp^{-1}(E)\cap D_\emptyset \cap U_m \cap \spt\gammap))\le \frac{1}{L_m^d}\HH^d(E)=0$.
 \end{proof}

We now turn to the case $S\neq \emptyset$, where the following holds:
 \begin{lemma}\label{lem:diameterDS2}
Let $S\subset\{1,\dots, N\}$ be such that $S\neq\emptyset$.

If $E\subset\RR^d$ is such that $\diam(E)<\delta$ for some $\delta>0$, then 
 \[\diam(\pi^i(\xp^{-1}(E)\cap D_S) )<\delta \quad \text{for every} \quad i\in S.\]
 In particular, if $E\subset\RRd$ is such that $\LL^d(E)=0$, then 
    \begin{equation*}\label{eq:hausdorffzero2}
     \LL^d(\pi^i(\xp^{-1}(E)\cap D_S))=\HH^d(\pi^i(\xp^{-1}(E)\cap D_S))=0 \quad \text{for every} \quad i\in S.
 \end{equation*}
 \end{lemma}
 Note that this result does not depend on the plan $\gamma_p$, but only on properties of the function $\xp$ on the set $D_S$.
 \begin{proof}
  Let $\widetilde E \coloneq \xp^{-1}(E)$ and take $w,\widetilde w\in\widetilde E$. Then there exist $z,\widetilde z \in E$ such that $w\in \xp^{-1}(\{z\})\cap D_S$ and $\widetilde w\in \xp^{-1}(\{\widetilde z\})\cap D_S$. By definition \eqref{eq:defDS} of $D_S$, it follows that 
 \[|\pi^i(w)-\pi^i(\widetilde w)|= |z-\widetilde z|\le \delta \quad \text{for every} \quad i\in S. \]

The proof of the second part is analogous to the one of Lemma~\ref{lem:diameterDempty}.
 \end{proof}

 \begin{theorem}\label{th:absolutecontinuitypsmall}
Assume that  $\mu_1,\dots,\mu_N\ll\LL^d$. Then, if $\spt\gamma_p$ is $c_p$-monotone, 
 \begin{equation*}
     \nu_p:=(\xp)_{\sharp}\gamma_p\ll\LL^d.
 \end{equation*}
\end{theorem}
\begin{proof}
   Let us consider a set $E$ such that $\LL^d(E)=0$. Then, given the countable cover $\{U_m\}_{m\in\NN}$ of $D_\emptyset\cap\gamma_p$ defined in Proposition \ref{prop:Dzerocase},
 \begin{align*}
(\xp)_\sharp\gamma_p(E)=&\gamma_p\left(\xp^{-1}(E)\cap\spt\gamma_p\cap \bigcup_{S\subset \{1,\dots,N\}}D_S\right)\\&\le\gamma_p\left(\xp^{-1}(E)\cap\spt\gamma_p\cap D_\emptyset\right)+\sum_{S\neq\emptyset}\gamma_p\left(\xp^{-1}(E)\cap\spt\gamma_p\cap D_S\right)\\
&\le \gamma_p\left(\xp^{-1}(E)\cap\spt\gamma_p\cap D_\emptyset\cap \bigcup_{m\in\NN}U_m\right)+\sum_{S\neq\emptyset}\gamma_p\left(\xp^{-1}(E)\cap\spt\gamma_p\cap D_S\right)\\&\le \sum_{m\in\NN}\gamma_p(\xp^{-1}(E)\cap D_\emptyset\cap U_m\cap\spt\gamma_p)+\sum_{S\neq \emptyset}\gamma_p\left(\xp^{-1}(E)\cap\spt\gamma_p\cap D_S\right)\\ &\le\sum_{m}\mu_{1}(\pi^{1}(\xp^{-1}(E)\cap D_\emptyset\cap\spt\gamma_p\cap U_m))+\sum_{S\neq \emptyset}\mu_{i_S}\left(\pi^{i_S}(\xp^{-1}(E)\cap\spt\gamma_p\cap D_S)\right),
 \end{align*}
 where $i_S\in S$ for every $S\neq\emptyset$. The last inequality is due to the marginal constraint on the transport plan $\gammap$.
 Thanks to Lemma \ref{lem:diameterDempty}, $\LL^d(\pi^{1}(\xp^{-1}(E)\cap D_\emptyset\cap\spt\gamma_p\cap U_m))=\mathcal{H}^d(\pi^{1}(\xp^{-1}(E)\cap D_\emptyset\cap\spt\gamma_p\cap U_m))=0$ for every $m\in\NN$, and by Lemma \ref{lem:diameterDS2},  $\LL^d(\pi^{i_S}(\xp^{-1}(E)\cap\spt\gamma_p\cap D_S))=\mathcal{H}^d(\pi^{i_S}(\xp^{-1}(E)\cap\spt\gamma_p\cap D_S))=0$ for every $i_S\in S$ and for every $S\subset\{1,\dots,N\}$, $S\neq\emptyset$. We conclude thanks to the absolute continuity of $\mu_1,\dots, \mu_N$.
\end{proof}

\subsection{Absolute continuity of the $W_p$-barycenter: the case \texorpdfstring{$p\ge2$}{p>=2}}  
As pointed out in Remark~\ref{rmk:singularities}, $\hp\in \mathcal{C}^2(\RRd)$ if $p\ge 2$. This observation allows us to study a reduced barycenter problem on each set $D_S$ with $S\neq\{1,\dots,N\}$, where $D^2 \hp$ does not degenerate and thus an injectivity estimate analogous to the one of Proposition~\ref{prop:Dzerocase} holds.
\begin{lemma}\label{lem:reducedproblem}
Let $p\ge2$ and $S\subset\{1,\dots, N\}$, $S\neq\{1,\dots, N\} $. Then for every $\xvec=(\xdots)\in D_S$, $x_i$ is a solution of the variational problem 
\begin{equation}\label{eq:reducedbary}
    {\xp}^{S^\mathsf{c}}(\xvec_{S^\mathsf{c}}):=\underset{z\in\RR^d}\argmin\sum_{j\not\in S}\lambda_j|x_j-z|^p \quad \text{for every}\quad  i\in S,
\end{equation}
 where, if $|S|=K<N$, $\xvecS\in \RR^{(N-K)d}$ is the vector with components $x_j$, $j\not\in S$. 
\end{lemma}
\begin{proof}
By optimality, the unique solution of \eqref{eq:reducedbary} is the unique point $z\in\RR^d$ such that 
\begin{equation}\label{eq:firstoptimalitybaryreduced}
    \sum_{j\not\in S}\lambda_jD\hp(x_j-z)=0.
\end{equation}
Moreover, by optimality of $\xp(\xvec)$,
\begin{equation}\label{eq:firstoptimalitybaryreduced-2}
    \sum_{j=1}^N\lambda_jD\hp(x_j-\xp(\xvec))=0.
\end{equation}
By definition of $D_S$ and by \eqref{eq:singularityinxiequalbarygradient}, $D\hp(x_i-\xp(\xvec))=0$ for every $i\in S$. Thus, \eqref{eq:firstoptimalitybaryreduced-2} becomes
\begin{equation*}
    \sum_{j\not\in S}\lambda_jD\hp(x_j-\xp(\xvec))=0,
\end{equation*}
and thus $\xp(\xvec)={\xp}^{S^\mathsf{c}}$. We conclude by recalling that $x_i=\xp(\xvec)$ for every $i\in S$ by definition of $D_S$.
\end{proof}
\begin{proposition}\label{prop:DScase}
Let $p\ge 2$, let $\spt\gamma_p$ be $c_p$-monotone and $S\subset\{1,\dots, N\}$, $S\neq\{1,\dots, N\}$. Then there exists a countable cover $\{U_m\}_{m\in\NN}$ of the set $D_S\cap\spt\gamma_p$ with the following property: For every $m\in\NN$, there exists $L_m>0$ such that 
 \begin{equation}\label{eq:inverselipschitzp}
     |\xp(\yvec)-\xp(\ytildevec)|\ge L_m\left(\sum_{j\not\in S}|y_j-\widetilde y_j|^2\right)^{\frac 12}
 \end{equation}
 for every $\yvec=(\ydots),\ytildevec=(\ytildedots)\in D_S\cap\spt\gamma_p\cap U_m$.
\end{proposition}
\begin{proof}
Let $\xvec=(\xdots)\in D_S$. For any fixed $j_0\not\in S$, we can apply Lemma~\ref{lem:firstinequality}, which implies that for any given $\eps>0$ there exists $r>0$ such that inequality \eqref{eq:firstinequality} holds, i.e. 
for every $\yvec, \ytildevec\in B(\xvec, r)\cap\spt\gamma_p$
\begin{equation}\label{eq:firstinequalityp}
  (y_{j_0}-\widetilde{y}_{j_0})^\intercal\left(D^2\hp(x_{j_0}-\xp(\xvec))\right) \left(\xp(\yvec)-\xp(\ytildevec)\right)\ge \Lambda_{j_0} |y_{j_0}-\widetilde y_{j_0}|^2-\eps C_{j_0}|\yvec-\ytildevec|^2.
\end{equation}
Now, by Lemma \ref{lem:reducedproblem} we know that for every $i\in S$, $x_i=\xpS(\xvecS)$, where $\xpS$ is defined in \eqref{eq:reducedbary}, and that the function $\xpS:\RR^{(N-K)d}\to\RR^d$, with $K:=|S|<N$, is the only solution of \eqref{eq:firstoptimalitybaryreduced}.  Moreover for any $j\not\in S$, by \eqref{eq:singularityinxiequalbary} in Remark \ref{rmk:singularities}, $\det (D^2\hp(x_{j}-\xp(x)) )\neq 0$. Thus we can use Lemma \ref{lem:diffbary} to obtain the existence of an open neighborhood $U_{\xvecS}\subset\RR^{(N-K)d}$ of $\xvecS$ where $\xpS \in \mathcal{C}^1$ and therefore locally Lipschitz. Thus for $\tau>0$, such that $B(\xvecS,\tau)\Subset U_{\xvec_S} $, 
\begin{equation*}
    |\xpS(\yvecS)-\xpS(\ytildevecS)|^2\le L_S\sum _{j\not\in S}|y_j-\widetilde y_j|^2
\end{equation*}
for every $\yvec_S,\ytildevec_S\in B(\xvec_S,\tau)$,
where $L_S$ is the Lipschitz constant of $\xpS$ on $B(\xvec_S,\tau)$. This implies that if one chooses $r>0$ (possibly smaller) such that $\pi_{S^{\mathsf{c}}}(B(\xvec,r))\subset B(\xvec_S,\tau),$\footnote{Here $\pi_{S^{\mathsf{c}}}:\RR^{Nd}\to \RR^{(N-K)d}$, such that for $\mathbf{y}\in\RR^{Nd}$, $\pi_{S^{\mathsf{c}}}(\yvec)$ is the vector with components $y_j$, $j\notin S$.}
\begin{align}\label{eq:reducingnorm}
  |\yvec-\ytildevec|^2=\sum_{j=1}^N|y_j-\widetilde y_j|^2&= \sum_{j\not\in S}|y_j-\widetilde y_j|^2+\sum_{j\in S}|y_j-\widetilde y_j|^2\\
  \notag&= \sum_{j\not\in S}|y_j-\widetilde y_j|^2+\sum_{j\in S}|\xpS(\yvecS)-\xpS(\ytildevecS)|^2\le (1+ KL_s)\sum_{j\not\in S}|y_j-\widetilde y_j|^2,
\end{align}
for every $\yvec,\ytildevec\in B(\xvec,r)\cap D_S$, where the third equality follows from Lemma~\ref{lem:reducedproblem}.\\
By plugging equation \eqref{eq:reducingnorm} into equation \eqref{eq:firstinequalityp}, we get 
\begin{equation}\label{eq:newfirstineq}
  (y_{j_0}-\widetilde{y}_{j_0})^\intercal D^2 h_p(x_{j_0}-\xp(\xvec)) \left(\xp(\yvec)-\xp(\ytildevec)\right)\ge \Lambda_{j_0} |y_{j_0}-\widetilde y_{j_0}|^2-\eps C_{j_0}(1+ KL_s)\sum_{j\not\in S}|y_j-\widetilde y_j|^2
\end{equation}
for every $\yvec,\ytildevec\in B(\xvec,r)\cap D_S$, where $C_{j_0}=N(1+M_{i_0}(\xvec))>0$.

Applying the same reasoning to every $j\not\in S$ and summing \eqref{eq:newfirstineq} over all $j\notin S$, we obtain 
\begin{equation*}
    \sum_{j\not\in S}(y_{j}-\widetilde{y}_{j})^\intercal D^2\hp(x_{j}-\xp(\xvec))\left(\xp(\yvec)-\xp(\ytildevec)\right)\ge \sum_{j\not\in S} \Lambda_{j} |y_{j}-\widetilde y_{j}|^2-\eps C(1+ KL_s)\sum_{j\not\in S}|y_j-\widetilde y_j|^2,
\end{equation*}
where $C=(N-K)\max_{j\not\in S}C_j$.
Therefore, by the Cauchy-Schwarz inequality,
\begin{equation*}
    K^{\frac 12}\left(\sum_{j\not\in S}|y_{j}-\widetilde{y}_{j}|^2\right)^{\frac 12 }M|\xp(\yvec)-\xp(\ytildevec)|\ge \sum_{j\not\in S} \Lambda_S |y_{j}-\widetilde y_{j}|^2-\eps C(1+ KL_s)\sum_{j\not\in S}|y_j-\widetilde y_j|^2,
\end{equation*}
where $M\coloneq\max_{j\not\in S}|D^2\hp(x_{j}-\xp(\xvec)|$ and $\Lambda_S\coloneq\min_{j\not\in S} \Lambda_j$.
Choosing $\eps$ small enough and a suitable $r(\xvec)>0$, we conclude the existence of $L(\mathbf{x})>0$ such that \begin{equation*}|\xp(\yvec)-\xp(\ytildevec)|\ge L(\xvec)\sum_{j\not\in S}|y_j-\widetilde y_j|
 \end{equation*}
 for every $\yvec,\ytildevec\in B(\xvec,r(\xvec))\cap D_S\cap\spt\gamma_p$.
Clearly, $\spt\gamma_p\cap D_S\subset\bigcup_{\xvec\in D_S}B(\xvec,r(\xvec))$. As
 every subset of $\RR^{Nd}$ is second countable, we can extract countably many points $\{\xvec_m\}\subset D_S$ such that $D_S\cap\spt\gamma_p\subset\bigcup_{m\in\NN}B(\xvec_m,r_m)$. The claim then follows with $U_m\coloneq B(\xvec_m,r_m)$.
\end{proof}
\begin{lemma}\label{lem:diameterDS1}
 Let $\spt\gamma_p$ be $c_p$-monotone, let $S\subset\{1,\dots,N\}$ be such that $S\neq\{1,\dots,N\}$, and let $\{U_m\}_{m\in\NN}$ be the countable cover of $D_S\cap\spt\gamma_p$ defined in Proposition \ref{prop:DScase}.
 
If $E\subset\RR^d$ is such that $\diam(E)<\delta$ for some $\delta>0$, then 
 \[\diam(\pi^j(\xp^{-1}(E)\cap D_S\cap U_m\cap\spt\gamma_p) )<\frac {\delta} {L_m}\]
 for every $j\not\in S$ and for every $m\in\NN$.
 
 In particular, if $E\subset\RRd$ is such that $\LL^d(E)=0$, then
    \begin{equation*}
     \LL^d(\pi^j(\xp^{-1}(E)\cap D_s\cap U_m \cap\spt\gammap))=\HH^d(\pi^j(\xp^{-1}(E)\cap D_s\cap U_m\cap\spt\gammap))=0,
 \end{equation*}
 for every $j\not\in S$ and for every $m\in\NN$.
 \end{lemma}
 \begin{proof}
 Let us fix $U_m$ and call $\widetilde E = \xp^{-1}(E)\cap D_S\cap\spt\gammap\cap U_m$. Take $w,\widetilde w\in\widetilde E$, then there exist $z,\widetilde z \in E$ such that $w\in \xp^{-1}(\{z\})\cap D_S\cap\spt\gamma_p\cap U_m$ and $\widetilde w\in \xp^{-1}(\{\widetilde z\})\cap D_S\cap\spt\gamma_p\cap U_m$. By Proposition~\ref{prop:DScase}, we have that for every $j\not\in S$
 \[|\pi^j(w)-\pi^j(\widetilde w)|\le\left(\sum_{j\not\in S}|\pi^j(w)-\pi^j(\widetilde w)|^2\right)^{\frac 12}\le \frac{1}{L_m} |\xp(w)-\xp(\widetilde w)|= \frac{1}{L_m} |z-\widetilde z|\le \frac {\delta} {L_m}.\]
The proof of the second part is analogous to the one of Lemma \ref{lem:diameterDempty}.
 \end{proof}

\begin{theorem}\label{th:absolutecontinuitypbig}
Assume that $\mu_1\ll \LL^d$. Then, if $\spt\gamma_p$ is $c_p$-monotone, 
 \begin{equation*}
     \nu_p:={\xp}_{\sharp}\gamma_p \ll \LL^d.
 \end{equation*}
\end{theorem}
Notice that by Proposition \ref{prop:equivalence}, $\nup$ is a solution of \eqref{C2Mpbary}.

\begin{proof}
   Let us consider a set $E$ such that $\LL^d(E)=0$, and define the family $\mathcal{F}_1$ of subsets of  $\{1,\dots, N\}$ that contain $1$, i.e.\ 
\begin{equation}\label{eq:F1}
    \mathcal{F}_1:=\{S\subset \{1,\dots, N\} \, : \, 1 \in S \}.
\end{equation}
Then,
 \begin{align*}
(\xp)_\sharp\gamma_p(E)&=\gamma_p\left(\xp^{-1}(E)\cap\spt\gamma_p\cap \bigcup_{S\subset \{1,\dots,N\}}D_S\right)\\&\le\sum_{S\not\in \mathcal{F}_1}\gamma_p\left(\xp^{-1}(E)\cap\spt\gamma_p\cap D_S\right)+\sum_{S\in \mathcal{F}_1}\gamma_p\left(\xp^{-1}(E)\cap\spt\gamma_p\cap D_S\right)\\
&\le\sum_{S\not\in \mathcal{F}_1}\mu_1\left(\pi^1(\xp^{-1}(E)\cap\spt\gamma_p\cap D_S)\right)+\sum_{S\in \mathcal{F}_1}\mu_1\left(\pi^1(\xp^{-1}(E)\cap\spt\gamma_p\cap D_S)\right)\\
&\le\sum_{S\not\in \mathcal{F}_1}\sum_{m\in\NN}\mu_1\left(\pi^1(\xp^{-1}(E)\cap\spt\gamma_p\cap D_S\cap  U^S_m)\right) +\sum_{S\in \mathcal{F}_1}\mu_1\left(\pi^1(\xp^{-1}(E)\cap\spt\gamma_p\cap D_S)\right),
 \end{align*}
 where for each $S\not\in \mathcal{F}_1$, $\{U^S_m\}$ is the countable cover of $D_S\cap\spt\gamma_p$ defined in Proposition \ref{prop:DScase}.
Notice that the second inequality is due to the marginal constraint $\pi^1_\sharp\gamma_p=\mu_1$. By Lemma~\ref{lem:diameterDS1}, $\LL^d(\pi^1(\xp^{-1}(E)\cap\spt\gamma_p\cap D_S\cap U^S_m))=\mathcal{H}^d(\pi^1(\xp^{-1}(E)\cap\spt\gamma_p\cap D_S \cap U^S_m))=0$ for every $S\not\in \mathcal{F}_1$ and every $m\in \NN$. Lemma~\ref{lem:diameterDS2} then gives $\LL^d(\pi^1(\xp^{-1}(E)\cap\spt\gamma_p\cap D_S))=\mathcal{H}^d(\pi^1(\xp^{-1}(E)\cap\spt\gamma_p\cap D_S))=0$ for every  $S\in\mathcal{F}_1$. We conclude thanks to the absolute continuity of $\mu_1$. 
\end{proof}

\subsection{Proof of \ref{th:sparsity-A} of Theorem \ref{th:sparsityplan}}
    Let $\gammap$ be an optimal coupling for the problem \eqref{MMpbary}, then by Corollary~\ref{cor:optimalitytwomarginals}, the couplings $\gamma_i:=(\pi^i,\xh)_\sharp\overline\gammap$ are optimal for $W_p(\mu_i,\nup)$, $i=1, \dots, N$. As the cost function $c_p$ is continuous (see Corollary \ref{cor: continuity bary}), Proposition \ref{prop:optimalplanarecmonotone} yields that $\gammap$ is $c_p$-monotone. Thus, by Theorem~\ref{th:absolutecontinuitypbig}, $\nup=(\xp)_\sharp\gammap$ is absolutely continuous. The Gangbo--McCann Theorem \cite[Theorem 1.2]{GMC96} then ensures that there exist unique maps $f_i$ such that $\mu_i=(f_i)_\sharp\nup$ and $\gamma_i=(f_i,\id)_\sharp\nup$ for every $i={1,\dots,N}$. This implies that 
\begin{equation*}
\gammap=(f_1,\dots,f_N)_\sharp\overline\nu.
\end{equation*}
Now, as $\mu_1 \ll \LL^d$, by the same result of Gangbo--McCann we know that there exists an optimal map $g_1$ such that $\gamma_1=(\id,g_1)_\sharp\mu_1$, and $g_1$ is the a.e.-inverse of $f_1$. Hence,
\begin{equation*}
 \gammap=(f_1,\dots,f_N)_\sharp(g_1)_\sharp\mu_1=(\id,f_2\circ g_1, \dots, f_N\circ g_1)_\sharp\mu_1.
\end{equation*}
Once the Monge structure of the optimal transport plan has been established, the almost-everywhere uniqueness of the maps $T_i = f_i \circ g_1$ then follows from a standard argument: the convex combination of two different optimal plans is still optimal, but cannot be Monge, which contradicts what we just showed. \hfill $\blacksquare$

\section{Optimality system}
\label{sec:p-partB} 

It is well-known that, as in the two-marginal case, multi-marginal OT admits a dual formulation. The attainment of dual optimizers (Kantorovich potentials) in the space of integrable (w.r.t.\ the marginal measure $\mu_i$) is a classical result \cite{K84}. 
Contrary to Kellerer's proof, which relies on weak compactness in $L^1_{\mu_i}$, we extend the strategy of \cite{AC11} from the case $p=2$ to general $1<p<\infty$ by use of the $p$-transform and its regularizing properties, see Lemma~\ref{lem:equi-Lipschitz} below. This gives us compactness and allows us to choose optimal potentials that are in addition continuous and almost everywhere differentiable on the convex hull of their supports, which will be needed in the proof of \ref{th:sparsity-B} of Theorem \ref{th:sparsityplan}. 

We therefore state the Kantorovich duality specifically for our problem \eqref{MMpbary}. 
For this purpose, we introduce the space \[\mathcal{Y}_p \coloneq (1+|\cdot|^p)\mathcal{C}_b(\RR^d) \coloneq \left\{ \varphi \in \mathcal{C}(\RR^d): (1+|\cdot|^p)^{-1}\varphi \text{ is bounded}\right\}\] of continuous functions of at most $p$-growth and we define the  $\lambda h_p$-conjugate of a function $\varphi:\RR^d \to \RR$ via 
\begin{equation}\label{eq:duality relaxed}
    \varphi^{\lambda, p}(x)\coloneq\inf_{z\in \RRd}\left\{\lambda |x-z|^p - \varphi(z) \right\}=\inf_{z\in \RRd}\left\{\lambda h_p(x-z) - \varphi(z) \right\} \quad \text{for } x\in\RRd.
\end{equation}

\begin{theorem}[MMOT Duality]
    \label{th:duality}
There holds
\begin{equation}\label{Dpbary} \tag{D-$p$-bar}
    C_{p-\text{MM}}=\sup_{\varphi_1,\dots,\varphi_N\in\mathcal{A}(c_p)}\sum_{i=1}^{N}\int_{\RRd}\varphi_id\mu_i,
\end{equation}
where 
\[\mathcal{A}(c_p)\coloneq\{ (\varphi_1,\dots,\varphi_N)\in L_{\mu_1}^1(\RRd)\times \cdots \times L_{\mu_N}^1(\RRd) \, : \, \varphi_1\oplus\cdots\oplus\varphi_N\le c_p\}. \]
Moreover, there exists a maximizer $\Phi=(\widetilde\varphi_1,\dots,\widetilde\varphi_N)\in\mathcal{B}(c_p)$, where
\begin{align*}
    \mathcal{B}(c_p)\coloneq\left\{ (\varphi_1,\dots,\varphi_N) \in\mathcal{A}(c_p) : \, \varphi_i= \psi_i^{\lambda_i, p} \text{ with } \psi_i\in \mathcal{Y}_p \text{ for every } i, \, \text{and} \, \sum_{i=1}^N \psi_i = 0   \right\}.
\end{align*}
In particular, $\widetilde\varphi_i$ is $\LL^d$-a.e. differentiable on the convex hull of the support of $\mu_i$ for every $i=1,\dots,N$. 
\end{theorem}
Existence of a maximizer $\Phi$ in \eqref{Dpbary} relies on the following quantitative version of a classical regularity result on the $\lambda h_p$-conjugate of a function based on the convexity of $h_p$ (see \cite[Corollary C.5]{GMC96}): 
\begin{lemma}\label{lem:equi-Lipschitz}
    Let $1<p<\infty$ and $\mathcal{F}$ be a family of functions $\varphi:\RR^d \to \RR$ such that the $\lambda h_p$-conjugates $\{\varphi^{\lambda, p}\}_{\varphi\in\mathcal{F}}$ satisfy, for any $R>0$,
    \begin{align}\label{eq:c-transform-bdd}
        \sup_{B_{2R}}|\varphi^{\lambda,p}| \leq M_R ,
    \end{align}
    with a constant $M_R$ depending only on $d,\lambda,p$, and $R$. Then the family $\{\varphi^{\lambda,p}\}_{\varphi\in\mathcal{F}}$ is equi-Lipschitz on $B_{R}$.
\end{lemma}
\begin{remark}
 For $p=2$, the equi-Lipschitzianity of the family $\{\varphi^{\lambda, 2}\}_{\varphi\in\mathcal{F}}$ in Lemma~\ref{lem:equi-Lipschitz} follows more easily from the observation that the functions $\lambda |\cdot|^2 - \varphi^{\lambda,2}$ are convex, as used in \cite{AC11}. Indeed, 
    \begin{align*}
        \lambda |z|^2 - \varphi^{\lambda,2}(z) = \sup_{x\in\RR^d} \left\{ 2\lambda x \cdot z - \lambda |x|^2 + \varphi(z) \right\},
    \end{align*}
    where the right hand side is convex as the supremum of linear (in $z$) functions indexed by $x\in\RR^d$. The equi-Lipschitzianity then follows from the fact that a family of convex functions on $B_{2R}$ is equi-Lipschitz on $B_{R}$, see e.g.\  \cite[Theorem 6.7]{EG15}.
\end{remark}
Before giving the proof of Lemma~\ref{lem:equi-Lipschitz} at the end of this section, let us show how it implies the duality Theorem~\ref{th:duality}. 
\begin{proof}[Proof of Theorem~\ref{th:duality}]
    Since $\mu_1,\dots, \mu_N\in \mathcal{P}_p(\RRd)$, we have that $C_{p-\text{MM}}<+\infty$ (see Footnote \ref{foot:finitecost}). The equivalence \eqref{Dpbary} of the two problems is standard, see e.g.\ \cite[Theorem 2.21]{K84} or \cite[Theorem 3.4]{F24}.

For the existence of a maximizer we will show that any maximizing sequence  $(\Phi^k)_{k\in\NN}=(\varphi^k_1,\dots,\varphi^k_N)_{k\in \NN}\subset\mathcal{A}(c_p)$ can be chosen in such a way that it lies in $\mathcal{B}(c_p)$.
Next, we establish the uniform bound \eqref{eq:c-transform-bdd} for $\{\Psi^k = (\psi_1^k, \dots, \psi_N^k)\}_{k\in\NN}$ and apply Lemma~\ref{lem:equi-Lipschitz} to obtain compactness of the sequence $\{\Psi^k\}_{k\in\NN}$ by the Arzelà-Ascoli Theorem. The result then follows from a simple argument based on Fatou's Lemma. 

\begin{enumerate}[label=\textbf{Step \arabic*.},leftmargin=0pt,labelsep=*,itemindent=*,itemsep=10pt,topsep=10pt]
\item \label{step:1} Let us consider an $N$-tuple $(\varphi_1,\dots,\varphi_N)\in\mathcal{A}(c_p)$. 
Note that since the cost function $c_p$ is non-negative, the function $(0, \dots, 0)$ is a competitor for the problem \eqref{eq:duality relaxed}, so that we can assume without loss of generality that \begin{equation}\label{eq:sum integrals nonnegative}
    \sum_{i=1}^N\int_{\RRd}\varphi_i d\mu_i\ge 0.
\end{equation}
We show that there exists
 $(\widetilde\varphi_1,\dots,\widetilde\varphi_N)\in \mathcal{B}(c_p)$, such that 
 \[\varphi_1\oplus\cdots\oplus\varphi_N\le \widetilde\varphi_1\oplus\cdots\oplus\widetilde\varphi_N\le c_p.\]
Define the functions 
\begin{align*}
     \psi_i\coloneq \varphi_i^{\lambda_i,p} \quad \text{for} \quad i=1,\dots,N-1, \quad \text{and} \quad 
     \psi_N\coloneq-\sum_{i=1}^{N-1} \psi_i,
\end{align*}
and set  
\begin{equation}
     \label{eq:varphi i}\widetilde\varphi_i\coloneq \psi_i^{\lambda_i, p}, \quad i=1, \dots, N.
 \end{equation}
 For any $(\xdots)\in\RR^{Nd}$, using that $\sum_{i=1}^N \psi_i = 0$, the definition \eqref{eq:varphi i} of the functions $\widetilde\varphi_i$ implies that
 \begin{equation*}
      \sum_{i=1}^N\widetilde\varphi_i(x_i)=\sum_{i=1}^N\widetilde\varphi_i(x_i)+\sum_{i=1}^N \psi_i(z)\le \sum_{i=1}^N\lambda_i |x_i-z|^p, \ \text{for every }z\in\RRd, 
 \end{equation*}
hence
 \begin{equation*}
      \sum_{i=1}^N\widetilde\varphi_i(x_i)\le \inf_{z\in\RRd}\sum_{i=1}^N\lambda_i |x_i-z|^p=c_p(\xdots).
 \end{equation*}
Moreover for every $i=1,\dots,N-1$, $\widetilde\varphi_i=(\varphi_i^{\lambda_i, p})^{\lambda_i, p}$ and thus $\widetilde\varphi_i\ge \varphi$. When $i=N$, we have
\begin{align*}
     \widetilde\varphi_N(x_N)&=\inf_{z\in \RRd}\left\{\lambda_N |x_N-z|^p+\sum_{i=1}^{N-1}\psi_i(z)\right\}\\
     &=\inf_{z\in \RRd}\left\{\lambda_N |x_N-z|^p+ \sum _{i=1}^{N-1}\inf_{x_i\in\RRd}\left\{\lambda_i|x_i-z|^p-\varphi_i(x_i)\right\} \right\}\\
     &=\inf_{z\in \RRd}\, \inf_{x_1,\dots,x_{N-1}\in\RR^{(N-1)d}} \left\{\lambda_N |x_N-z|^p+ \sum _{i=1}^{N-1}\left(\lambda_i|x_i-z|^p-\varphi_i(x_i)\right) \right\}\\
     &=\inf_{x_1,\dots,x_{N-1}\in\RR^{(N-1)d}} \, \inf_{z\in \RRd} \left\{\sum _{i=1}^{N}\lambda_i|x_i-z|^p-\sum _{i=1}^{N-1}\varphi_i(x_i) \right\}\\
     &=\inf_{x_1,\dots,x_{N-1}\in\RR^{(N-1)d}} \left\{c_p(x_1,\dots,x_N)-\sum _{i=1}^{N-1}\varphi_i(x_i)\right\}
     \ge \varphi_N(x_N),
 \end{align*}
where the equality in the last line follows by the definition of $c_p$ and the final inequality by the duality constraint $\varphi_1\oplus\cdots\oplus\varphi_N\le c_p$ in $\mathcal{A}(c_p)$ .
Thus $\varphi_1\oplus\cdots\oplus\varphi_N\le\widetilde\varphi_1\oplus\cdots\oplus\widetilde\varphi_N\le c_p$.

\medskip
By subtracting $\psi_i(0)$ from $\psi_i$, we may assume that $\psi_i(0)=0$ for every $i=1,\dots,N$. Indeed, since $(\psi_i-\alpha)^{\lambda_i, p}=\psi_i^{\lambda_ip}+\alpha=\widetilde\varphi_i +\alpha$, we have $\widetilde\varphi_i=(\psi_i)^{\lambda_i, p}$ and, using that $\sum_i \psi_i(0)=0$, there still holds $\varphi_1\oplus\cdots\oplus\varphi_N\le\widetilde\varphi_1\oplus\cdots\oplus\widetilde\varphi_N=\widetilde\varphi_1+\psi_1(0)\oplus\cdots\oplus\widetilde\varphi_N + \psi_N(0)\le c_p$. 

Evaluating $\lambda_i|x_i-z|^p-\psi_i(z)$ in $z=0$, we therefore get
\begin{align}\label{eq:upperboundpotential1}
    \widetilde\varphi_i(x_i)\le\lambda _i |x_i|^p,  \quad \text
{for every } x_i\in\RRd, i=1,\dots,N. 
\end{align}
 Moreover, we observe that by \eqref{eq:upperboundpotential1}, using that $\mu_i \in \mathcal{P}_p(\RR^d)$,
 \begin{equation*}
    \sum_{i=1}^N\int _{\RRd}\widetilde\varphi_i \,\dd\mu_i  \le \sum_{i=1}^N \int_{\RR^d} \lambda_i|x_i|^p \,\dd\mu_i(x_i)=\sum_{i=1}^{N} \lambda_i K_i<+\infty,
 \end{equation*}
 where $K_i\coloneq|||\cdot|^p||_{L^1_{\mu_i}(\RRd)}$. Consequently, by \eqref{eq:sum integrals nonnegative} it follows that
 \begin{equation}\label{eq:lowerboundintegral}
   \int _{\RRd}\widetilde\varphi_i \,\dd\mu_i
   \ge -   \sum_{j\neq i}\int _{\RRd}\widetilde\varphi_j \,\dd\mu_j
   \ge -\sum_{j\neq i} \lambda_j K_j.
 \end{equation}
By integrating the inequality
 $\psi_i(z)\le \lambda_i |x_i-z|^p-\widetilde\varphi_i (x_i)$, which follows from the definition \eqref{eq:varphi i} of $\widetilde\varphi_i$, 
with respect to $\mathrm{d}\mu_i(x_i)$, 
 we obtain 
\begin{align}\label{eq:upperboundpotential}
    \notag \psi_i(z)&\le\int _{\RRd}\lambda_i |x_i-z|^p \,\dd\mu_i(x_i)-\int _{\RRd}\widetilde\varphi_i \,\dd\mu_i\\
    & \le \max\{1, 2^{p-1}\} \sum_{i=1}^N  \lambda_i K_i + 2^{p-1} \lambda_i |z|^p
    \leq C (1+ |z|^p), \quad \text{for every }z\in\RRd.
\end{align}
Using that $\sum_{i=0}^N \psi_i = 0$, we may also bound 
\begin{align*}
    \psi_i(z) = -\sum_{j\neq i} \psi_j(z) \stackrel{\eqref{eq:upperboundpotential}}{\geq} -\sum_{j\neq i} C (1+ |z|^p) \geq \widetilde{C} (1+|z|^p),
\end{align*}
hence there exists a constant $M$ depending only on $d, N, \lambda_i, K_i$ such that   
\begin{align}\label{eq:bound-psi}
    |\psi_i(z)|\leq M (1+|z|^p) \quad \text{for all } z\in\RR^d. 
\end{align}
Lemma~\ref{lem:equi-Lipschitz} then implies that the functions $\psi_i$ are continuous by for $i=1, \dots, N-1$, and therefore also $\psi_N = -\sum_{i=1}^{N-1}\psi_i$, which yields the claim $(\widetilde\varphi_1,\dots,\widetilde\varphi_N)\in \mathcal{B}(c_p)$.
 
\item \label{step:2} Let $(\widetilde{\Phi}^k)_{k\in\NN}\subset \mathcal{B}(c_p)$ be a maximizing sequence as constructed in \ref{step:1} Then the sequence $(\Psi^k)_{k\in\NN}$ lies in a precompact subset of $\mathcal{C}(\RR^d)^N$ by the Arzelà-Ascoli Theorem. Indeed, by Lemma~\ref{lem:equi-Lipschitz} and the bound \eqref{eq:bound-psi}, the sequences $\{\psi_i^k\}_{k\in\NN}$ are equi-Lipschitz on any ball $B_R(0)$, $R>0$, for $i=1, \dots, N-1$, and therefore also $\left\{\psi_N^k = -\sum_{i=1}^{N-1}\psi_i^k\right\}_{k\in\NN}$. We may hence extract a subsequence $\{\Psi^{n_k}\}_{k\in\NN}$ that converges uniformly on compact subsets to some $\Psi = (\psi_1, \dots, \psi_N)$ with $\psi_i \in \mathcal{Y}_p$ for any $i=1, \dots, N$. 

Define $\widetilde{\Phi} = (\widetilde{\varphi}_1, \dots, \widetilde{\varphi}_N)$ via $\widetilde{\varphi}_i \coloneq (\psi_i)^{\lambda_i, p}$ for $i=1, \dots, N$. 
By \eqref{eq:upperboundpotential1}, we know that $\widetilde\varphi_i(x_i)<+\infty$ for every $x_i$ and by \eqref{eq:lowerboundintegral} that $\widetilde\varphi_i(x_i)>-\infty$ for every $x_i\in\spt\mu_i$.  Being the $\lambda_ih_p$-conjugate of some function, by \cite[Proposition C.3]{GMC96} the function $\widetilde\varphi_i$ is locally bounded in the interior of the convex hull of $\spt\mu_i$. 
Appealing again to Lemma~\ref{lem:equi-Lipschitz}, it follows that the functions $\widetilde{\varphi}_i$ are locally Lipschitz continuous in the interior of the convex hull of $\spt\mu_i$, in particular differrentiable at $\LL^d$-almost every $x_i$ in the interior of the convex hull of $\spt\mu_i$ by Rademacher's Theorem.

\item \label{step:3} 
Note that for $x_i \in \RR^d$,
\begin{align}\label{eq:p-trafo-lim}
\begin{split}
    \widetilde{\varphi}_i(x_i)
    &= \inf_{z\in\RR^d} \left\{\lambda_i |x_i-z|^p - \psi_i(z) \right\}
    = \inf_{z\in\RR^d} \lim_{k\to\infty} \left\{\lambda_i |x_i-z|^p - \psi_i^{n_k}(z) \right\} \\
    &\geq \limsup_{k\to\infty} \inf_{z\in\RR^d} \left\{\lambda_i |x_i-z|^p - \psi_i^{n_k}(z) \right\}
    = \limsup_{k\to\infty} \left(\psi_i^{n_k}\right)^{\lambda_i, p}(x_i),
\end{split}
\end{align}
which implies with Fatou's Lemma (recall that $\int_{\RR^d} \varphi_i^n \,\mathrm{d}\mu_i \leq \lambda_i K_i < \infty$ for all $n\in\NN$) that 
\begin{align*}
    \sup_{\Phi \in \mathcal{A}(c_p)} \sum_{i=1}^N \int_{\RR^d} \varphi_i \,\mathrm{d}\mu_i 
    &= \sup_{\widetilde{\Phi} \in \mathcal{B}(c_p)} \sum_{i=1}^N \int_{\RR^d} \widetilde{\varphi}_i \,\mathrm{d}\mu_i 
    \geq \sum_{i=1}^N \int_{\RR^d} (\psi_i)^{\lambda_i,p} \,\mathrm{d}\mu_i \\
    &\stackrel{\eqref{eq:p-trafo-lim}}{\geq} \sum_{i=1}^N \int_{\RR^d} \limsup_{k\to\infty} (\psi_i^{n_k})^{\lambda_i,p} \,\mathrm{d}\mu_i 
    \geq \limsup_{k\to\infty} \sum_{i=1}^N \int_{\RR^d} (\psi_i^{n_k})^{\lambda_i,p} \,\mathrm{d}\mu_i  \\
    &= \limsup_{k\to\infty} \sum_{i=1}^N \int_{\RR^d} \widetilde{\varphi}^{n_k}_i\,\mathrm{d}\mu_i
    = \sup_{\widetilde{\Phi} \in \mathcal{B}(c_p)} \sum_{i=1}^N \int_{\RR^d} \widetilde{\varphi}\,\mathrm{d}\mu_i.
\end{align*}
It follows that $\widetilde{\Phi}\in\mathcal{B}(c_p)$ is a maximizer for \eqref{Dpbary}.
\qedhere
\end{enumerate}
\end{proof}

\begin{corollary}\label{cor: coupled two marg pot}
  For every $i=1,\dots,N$, the functions $(\widetilde\varphi_i, \psi_i)$, defined as in Theorem~\ref{th:duality}, are optimal potentials for $\lambda_iW_p(\mu_i,\nu_p)$, where $\nu_p=\mathrm{bar}_p((\mu_i,\lambda_i)_{i=1,\dots,N})$.
  \begin{proof}
We recall that by construction, for every $i=1,\dots,N$, $\widetilde\varphi_i$  is the $\lambda h_p$-conjugate of  $\psi_i$, and thus
    \begin{equation}\label{eq:ineq}
        \widetilde\varphi_i(x_i)+ \psi_i(z)\le \lambda_i |x_i-z|^p \quad \text{for every} \ (x_i,z) \in \RR^{2d}
    \end{equation}
    and that $\sum_{i=1}^N \psi_i(z)=0$, for every $z\in\RRd$.\\
    If $\gamma_p$ is the optimal plan for \eqref{MMpbary},  $\widetilde\varphi_1,\dots,\widetilde\varphi_N$ are the Kantorovich potentials in \eqref{Dpbary}, if and only if the nonnegative function $c_p-\widetilde\varphi_1\oplus\cdots\oplus\widetilde\varphi_N$ is equal $0$ on $\spt \gammap$. It follows that for every $\xvec \in\spt\gamma_p$,
\begin{align}\label{eq:equality on supp}
\sum_{i=1}^N \widetilde\varphi_i(x_i)+\sum_{i=1}^N \psi_i(\xp(\xvec))=\sum_{i=1}^N \widetilde\varphi_i(x_i)=\sum_{i=1}^N\lambda_i|x_i-\xp(\xvec)|^p.
\end{align}
By Proposition~\ref{prop:equivalence} there holds $\nu_p = {\xp}_\sharp\gammap$, and by Corollary~\ref{cor:optimalitytwomarginals} we know that the unique optimal plan $\gamma_i\in\Pi(\mu_i,\nu_p)$ for $\lambda_iW_p(\mu_i,\nu_p)$ is given by $\gamma_i=(\pi_i,\xp)_\sharp \gamma_p$. This implies that  $(x_i,z)\in\spt\gamma_i$ iff $z=\xp(\xvec)$, where the points $x_1,\dots, x_{i-1},x_{i+1},\dots, x_N$ are such that $\xvec=(x_1,\dots,x_N)\in\spt\gamma_p$. Thus, thanks to \eqref{eq:equality on supp}, inequality \eqref{eq:ineq} is actually an equality on the support of $\gamma_i$.
  \end{proof}
\end{corollary}

\begin{proposition}\label{prop:firstordoptsystem}
Let us assume that $\mu_1,\dots,\mu_N \ll \LL^d$.  Then, given the optimal plan $\gammap$ for \eqref{eq:MMhbary}, the first order optimality system
 \begin{equation}\label{eq:firstordoptsystem}
     \left\{\begin{aligned}
      \lambda_1Dh_p(x_1-\xp(\xvec))&=D\widetilde\varphi_1(x_1),\\
      &\,\,\,\vdots\\
      \lambda_NDh_p(x_N-\xp(\xvec))&=D\widetilde\varphi_N(x_N),\\
     \end{aligned}\right.
 \end{equation}
 holds for $\gammap$-a.e. $\xvec$,
 where $\widetilde\varphi_1,\dots,\widetilde\varphi_N$ are the Kantorovich potentials from Theorem~\ref{th:duality}.
\end{proposition}
\begin{proof}
Thanks to their optimality, the functions $\widetilde\varphi_1,\dots,\widetilde\varphi_N$ satisfy \eqref{eq:equality on supp}. In particular, for every fixed $\xvec=(\xdots)\in\spt\gammap$ with $z=\xp(\xvec)$, $x_i$ is a minimum of the nonnegative function \begin{align*}
    (x_1,\dots,x'_i,\dots,x_N,z)&\mapsto \lambda_i|x'_i-z|^p -\widetilde\varphi_i(x'_i)+\sum_{j\neq i}^N\lambda_j|x_j-z|^p-\sum_{j\neq i} \widetilde\varphi_j(x_j)\\&=\lambda_ih_p(x'_1-z)-\widetilde\varphi_i(x'_i)+\sum_{j\neq i}^N\lambda_jh_p(x_j-z)-\sum_{j\neq i}^N \widetilde\varphi_j(x_j),
\end{align*}
for every $i=1,\dots,N$.
The conclusion then follows from the differentiability of $h_p$ and the $\LL^d$-a.e. differentiability of $\widetilde\varphi_i$ on the interior of the convex hull of $\spt\mu_i$ for every $i=1,\dots, N$. Indeed, since the Lebesgue measure of the interior of a convex set is equal to the Lebesgue measure of the set and, since $\mu_i\ll \LL^d$, we have that $\widetilde\varphi_i$ is differentiable $\mu_i$-a.e.
\end{proof}

\begin{proof}[Proof of part \ref{th:sparsity-B} of Theorem~\ref{th:sparsityplan}].
Since the measures $\mu_1,\dots,\mu_N$ are all absolutely continuous w.r.t.\ $\LL^d$, the Kantorovich potentials $\widetilde\varphi_1, \dots, \widetilde\varphi_N$ (given by Theorem~\ref{th:duality}) satisfy the first-order optimality system by Proposition~\ref{prop:firstordoptsystem}. 
In particular, appealing to \eqref{eq:firstordoptsystem}, if $\xvec\in\spt\gammap$, 
\begin{equation*}
    Dh_p(x_i-\xp(\xvec))=p|x_i-\xp(\xvec)|^{p-2}(x_i-\xp(\xvec))=\lambda_i^{-1} D\varphi_i(x_i),
\end{equation*}hence
\begin{equation}\label{eq:baryfunctionx1}
    \xp(\xvec)=x_i-Dh_p^{-1}\left(\lambda_i^{-1} D\varphi_i(x_i)\right)=x_i-(p\lambda_i)^{-\frac{1}{p-1}} |D\varphi_i(x_i)|^{\frac{2-p}{p-1}} D\varphi_i(x_i)=:g_i(x_i)
\end{equation}
for every $i=1,\dots,N$, where $g_i$ is the optimal transport map (in the two-marginal sense) from $\mu_i$ to $\nup$. With the same argument as the one in the proof of \ref{th:sparsity-A} of Theorem \ref{th:sparsityplan}, we get that the functions $g_i$ are invertible for  $i=1,\dots,N$ (with an a.e.-inverse that we called $f_i$ in that proof), since $\nup=(\xp)_\sharp\gammap$ is absolutely continuous by Theorem~\ref{th:absolutecontinuitypbig}. Then
\begin{align*}
    x_i=g_i^{-1}(\xp(\xvec))=g_i^{-1}(g_1(x_1)), \quad \text{for every } i=2,\dots,N, 
\end{align*}
which, together with \eqref{eq:baryfunctionx1}, proves \eqref{eq:mapT-degenerate}.
\end{proof}

\begin{proof}[Proof of Lemma~\ref{lem:equi-Lipschitz}]
    Let $z\in B_{R}(0)$. We claim that there exists a radius $\widetilde{R} <\infty$ (depending only on $R, M_R, p, \lambda$) such that 
    \begin{align}\label{eq:inf-bounded}
        \varphi^{\lambda,p}(z) = \inf_{x\in B_{\widetilde{R}}(0)} \left(\lambda |x-z|^p - \varphi(x) \right)
    \end{align}
    for any $\varphi \in \mathcal{F}$. 
    Assuming the validity of \eqref{eq:inf-bounded}, the claim then follows from the the fact that $z\mapsto \lambda |x-z|^p$ is Lipschitz on $B_R(0)$ with Lipschitz constant bounded by $L_R \coloneq p\lambda (\widetilde{R}^{p-1} + R^{p-1})$, independent of $x\in B_{\widetilde{R}}(0)$. Hence, \eqref{eq:inf-bounded} implies that $\{\varphi^{\lambda,p}\}_{\varphi\in\mathcal{F}}$ is a family of Lipschitz continuous fuinctions on $B_{R}(0)$ with Lipschitz constant bounded by $L_R$. 
    
    It remains to prove \eqref{eq:inf-bounded}. To this end, take a minimizing sequence for $\varphi^{p,\lambda}(z)$, that is
    \begin{align*}
        \varphi^{\lambda,p}(z) = \lim_{n \to\infty} \left(\lambda |x_n - z|^p - \varphi(x_n)\right).
    \end{align*}
    If there exists $N\in\NN$ such that $|z-x_n| \leq R$ for all $n\geq N$, then $|x_n| \leq |z| + R \leq 2 R$. Hence we may take $\widetilde{R} = 2R$ in this case. 
    Let us therefore treat the case that along a subsequence (still denoted by $x_n$) there holds $|z-x_n| > R$. Set 
    \begin{align*}
        z_n \coloneq z - \frac{R}{2} \frac{z-x_n}{|z-x_n|}.
    \end{align*}
    Then $z_n - x_n = (1 - \frac{R}{2|z-x_n|}) (z-x_n)$ and $z-z_n = \frac{R}{2} \frac{z-x_n}{|z-x_n|}$, in particular $z_n \in B_{\frac{R}{2}}(z) \subset B_{2R}(0)$. By assumption \eqref{eq:c-transform-bdd} it follows that 
    \begin{align*}
        |\varphi^{\lambda,p}(z_n)| \leq M_R \quad \text{for all } n\in\NN. 
    \end{align*}
    Since $\{x_n\}_{n\in\NN}$ is a minimizing sequence, we may assume (by passing to a further subsequence if necessary) that for all $n\in\NN$ we have 
    \begin{align*} 
        \lambda |z-x_n|^p - \varphi(x_n) \leq \varphi^{\lambda,p}(z) + M_R \leq 2 M_R. 
    \end{align*}
    Moreover, by definition of the $\lambda h_p$-conjugate and the bound \eqref{eq:c-transform-bdd} there holds
    \begin{align*}
        \lambda |z_n-x_n|^p - \varphi(x_n) \geq \varphi^{\lambda,p}(z_n) \geq - M_R. 
    \end{align*}
    We conclude that 
    \begin{align}\label{eq:equi-L-lower-bound}
        |z-x_n|^p - |z_n - x_n|^p \leq \frac{3M_R}{\lambda} \quad \text{for all } z \in B_R(0) \text{ and } n \in \NN.
    \end{align}
    Convexity of $h_p = |\cdot|^p$ now implies that 
    \begin{align*}
        |z-x_n|^p - |z_n - x_n|^p 
        &\geq D h_p(z_n - x_n) (z-z_n) 
        = p |z_n-x_n|^{p-2} (z_n - x_n)\cdot (z-z_n) \\
        &= p |z_n-x_n|^{p-1} \left(1 - \frac{R}{2|z-x_n|}\right)^{p-1} \frac{R}{2} 
        \geq \frac{pR}{2^p} |z-x_n|^{p-1},
    \end{align*}
    where in the last estimate we used that $|z-x_n| > R$, and therefore $1 - \frac{R}{2|z-x_n|} > \frac{1}{2}$. Hence, together with \eqref{eq:equi-L-lower-bound} we obtain
    \begin{align*}
        |z-x_n| \leq \left( \frac{3\cdot 2^p M_R}{\lambda p R} \right)^{\frac{1}{p-1}},
    \end{align*}
    in particular $|x_n| \leq R + \left( \frac{3\cdot 2^p M_R}{\lambda p R} \right)^{\frac{1}{p-1}} \eqcolon \widetilde{R}$, which proves \eqref{eq:inf-bounded}.
\end{proof}
\section{\texorpdfstring{$p$}{p}-Wasserstein barycenters in one dimension}\label{sec:one-dim}
Here we provide a more detailed description of $p$-Wasserstein barycenters of probability measures on the line. In particular, we explain the statistical meaning of $p$-Wasserstein barycenters, give illustrative examples highlighting the role of the parameter $p$, and discuss the two natural limits $p\to 1$ and $p\to\infty$. 

\begin{figure}[H]
  \includegraphics[width=0.32\textwidth]{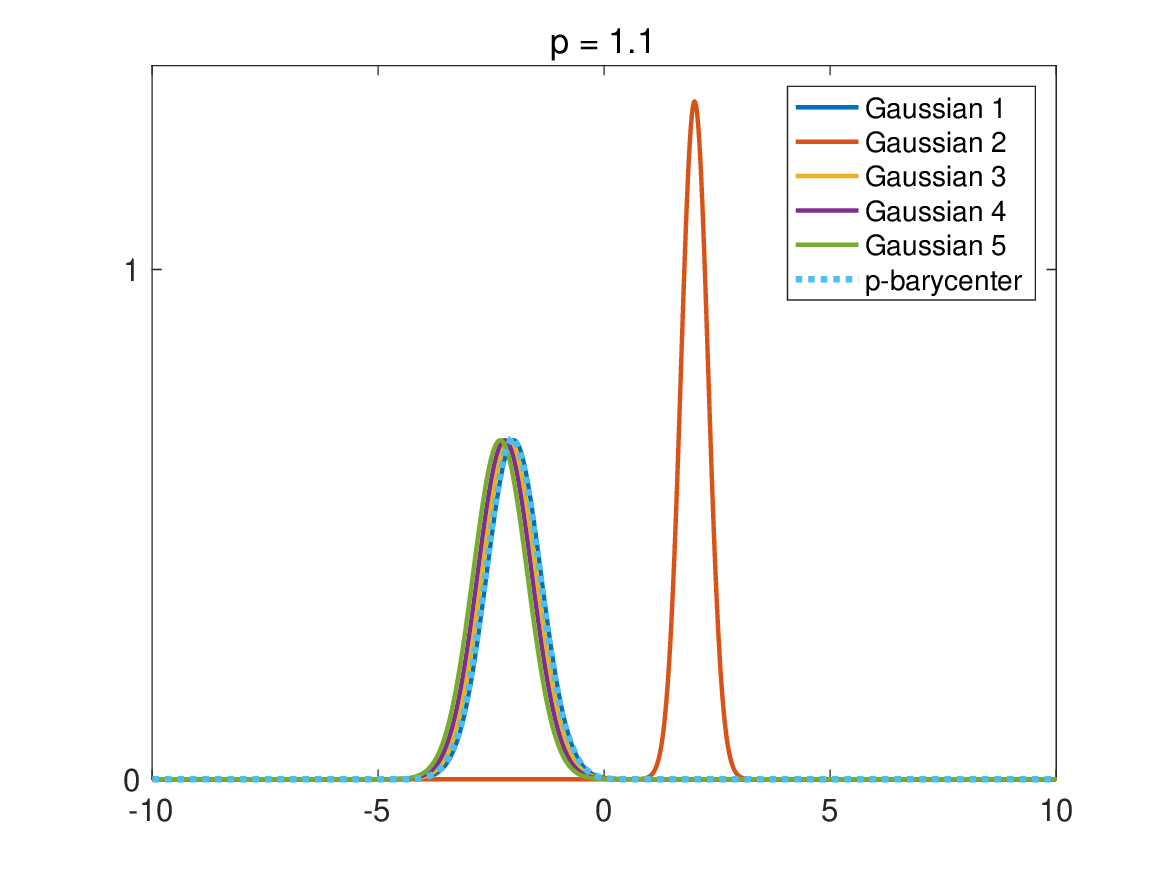}
    \includegraphics[width=0.32\textwidth]{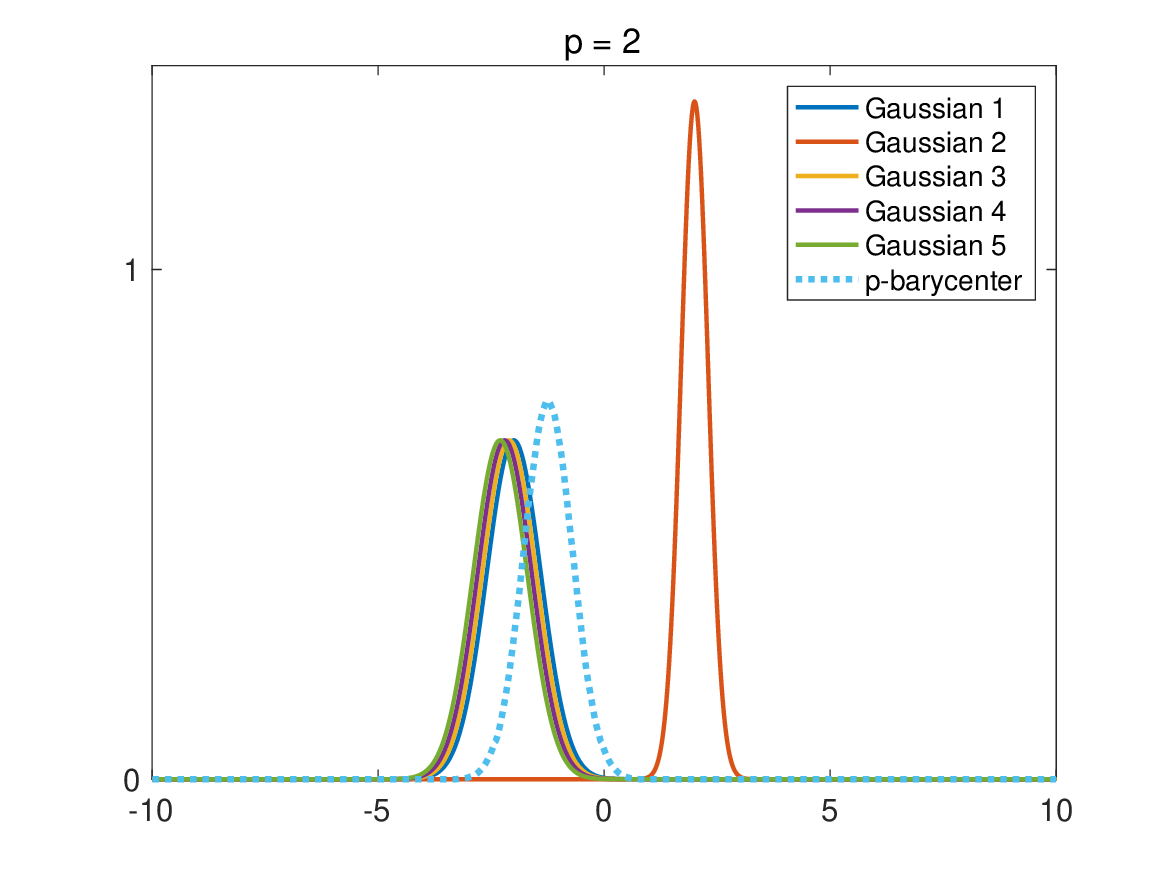}
    \includegraphics[width=0.32\textwidth]{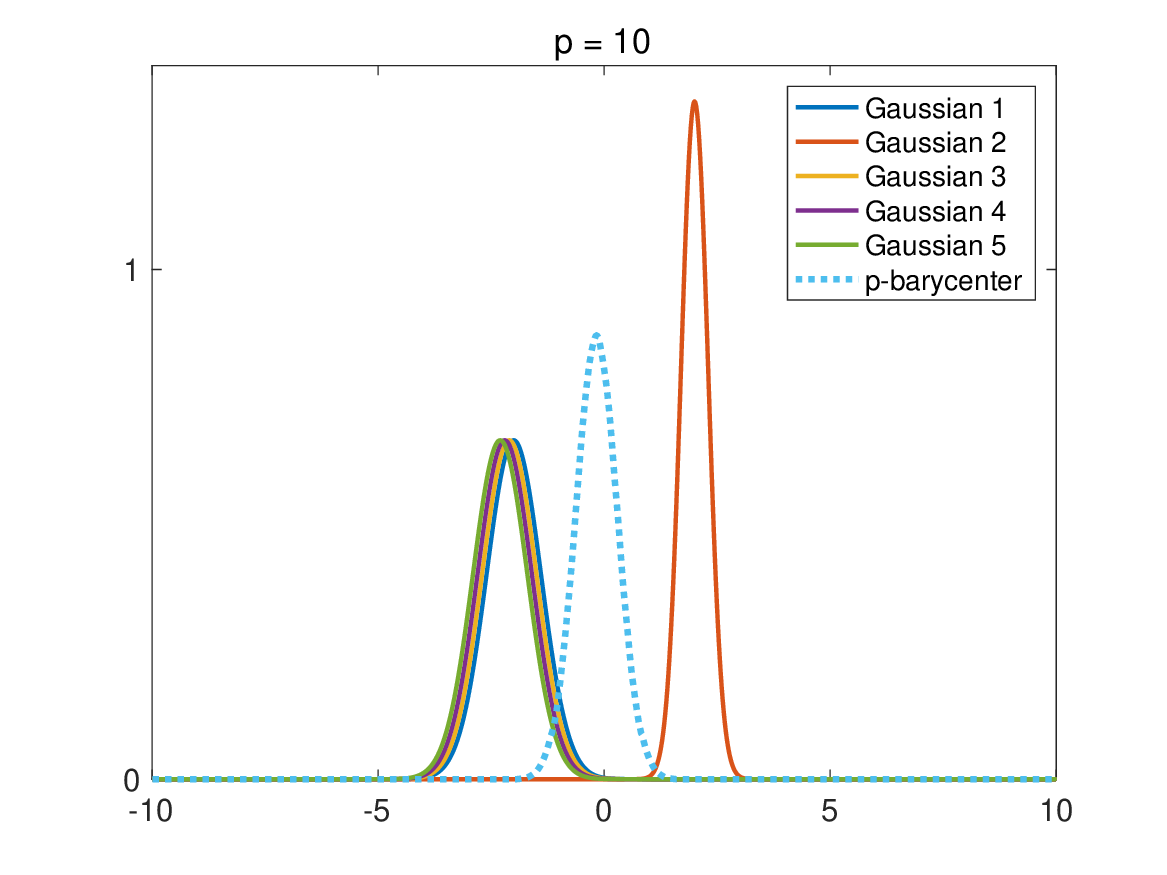}
    \caption{\footnotesize $p$-Wasserstein barycenter of five Gaussians, four of them similar and one very different, for different values of $p$. Note that for $p\approx 1$ the $p$-barycenter is insensitive to the very different Gaussian, whereas for $p\gg 1$ it is a half-and-half transport average between the two types of Gaussians which is insensitive to the fact that there is only one copy of the very different Gaussian.}
    \label{Fig:examples}
\end{figure}

Recall that for any probability measure on the line, i.e. $\mu\in\calP(\RR)$, the cumulative distribution function $F\, : \, \RR\to[0,1]$ is defined by 
$$
    F(x) := \mu((-\infty,x]),
$$
and its generalized inverse $F^{-1} \, : \, (0,1)\to\RR$ (the inverse distribution function) is defined by 
$$
   F^{-1}(t) := \inf\{y\in\RR\, : \, F(y)\ge t\}, \quad t\in(0,1).
$$
Recall also the statistical meaning of $Q=F^{-1}(y)$: $Q$ is the $y$-quantile of $\mu$, that is, the threshold point such that the $y^{\text{th}}$ part of the mass of $\mu$ lies to the left of $Q$ and the $(1-y)^{\text{th}}$ part lies to the right of $Q$. In particular, $F^{-1}(\frac{1}{2})$ is the median of $\mu$. 

It is well known that in one dimension the $p$-Wasserstein distance between two probability measures $\mu$ and $\nu$ agrees with the $L^p$ distance between their inverse distribution functions, 
\begin{equation} \label{pWass1D}
   W_p(\mu,\nu)^p = \int_0^1 |F^{-1}(y)-G^{-1}(y)|^p \,\dd y, \quad  1<p<\infty.
\end{equation}
This, together with Proposition \ref{prop:equivalence}, readily yields the following characterization of the $p$-Wasserstein barycenter: 
\begin{theorem}  \label{thm:1D}
Let $p\in(1,\infty)$. For any absolutely continuous probability measures $\mu_1,...,\mu_N\in\calP_p(\RR)$, their  $p$-Wassderstein barycenter is characterized by the property that 
\begin{equation} \label{1Deq}
   G_p^{-1}(y) = \xbar_p\left(F_1^{-1}(y),...,F_N^{-1}(y)\right) = \argmin_{z\in\RR} \sum_{i=1}^N \lambda_i|F_i^{-1}(y) - z|^p \; \forall y\in(0,1),
\end{equation}
that is to say its inverse distribution function at any point $y$ is the classical $p$-barycenter of the values of the inverse distribution functions of the $\mu_i$ at $y$. 
\end{theorem}
\begin{proof} By Proposition \ref{prop:equivalence} the inverse distribution function $G_p^{-1}$ of the $p$-barycenter minimizes the integral 
$$
  \int_0^1 \sum_{i=1}^N \lambda_i |F_i^{-1}(y)-G_p^{-1}(y)|^p dy.
$$
But the minimum over arbitrary measurable functions (not required to be any inverse distribution function) is obviously achieved by minimizing the integrand pointwise, and a straightforward adaptation of the argument in \cite{F24} for $p=2$ shows that the pointwise minimizer is in fact the inverse distribution function of some probability measure on $\RR$.
\end{proof}

From now on let us assume $\lambda_1=...=\lambda_N=\frac{1}{N}$. It is clear that the following limits of the function in \eqref{1Deq} exist: 
\begin{eqnarray} \label{1Deq1}
    G_1^{-1}(y) &= \;\lim\limits_{p\to 1} G_p^{-1}(y) &= \; F_{\mathrm{mid}}^{-1}(y) \;\;\;\;\; 
    \mbox{ when $N$ is odd } \\
    \label{1Deqinfty}
    G_\infty^{-1}(y) &= \;\lim\limits_{p\to\infty} G_p^{-1}(y) &= \;\frac{F_{\mathrm{max}}^{-1}(y) + F_{\mathrm{min}}^{-1}(y)}{2}.
\end{eqnarray}
Here for any $N$ real numbers $z_1,...,z_N$, when $z_{\sigma(1)}\le z_{\sigma(2)}\le ... \le z_{\sigma(N)}$ is their monotone ordering, $z_{\mathrm{min}}=z_{\sigma(1)}$ denotes their minimum, $z_{\mathrm{max}}=z_{\sigma(N)}$ denotes their maximum, and, for $N$ odd, $z_{\mathrm{mid}}=z_{\sigma(\tfrac{N+1}{2})}$ denotes their median.

As in the case $p\in(1,\infty)$, one can show that the functions in \eqref{1Deq1}, \eqref{1Deqinfty} are inverse distribution functions of unique probability measures $\nu_1$ respectively $\nu_\infty$, providing us with a unique definition of the $1$-Wasserstein barycenter and the $\infty$-Wasserstein barycenter,
\begin{eqnarray}
    && \mathrm{bar}_1((\mu_i)_{i=1,...,N}):=\nu_1, \\
    && \mathrm{bar}_\infty((\mu_i)_{i=1,...,N}) := \nu_\infty.
\end{eqnarray}
Equations \eqref{1Deq}, \eqref{1Deq1}, \eqref{1Deqinfty} provide us with the following statistical meaning of $p$-Wasserstein barycenters of a given collection $\mu_1,...,\mu_N$ of probability measures.
\begin{itemize}
    \item The quantile of the $p$-barycenter is the Euclidean $p$-barycenter of the quantiles of the $\mu_i$. 
    \item For $p=2$, the quantile of the $p$-barycenter is the average of the quantiles of the $\mu_i$. In particular, the $p$-barycenter shows some sensitivity to outliers.
    \item In the limit $p\to 1$, and for $N$ odd, the quantile of the $p$-barycenter approaches the median of the quantiles of the $\mu_i$. Thus the $p$-barycenter becomes less and less sensitive to outliers, instead reflecting the behaviour of `typical' $\mu_i$'s. 
    \item In the limit $p\to\infty$ the quantile of the $p$-barycenter approaches the arithmetic mean of the 
    largest and smallest quantile of the $\mu_i$. Thus the $p$-barycenter becomes more and more indicative of outliers, consisting of a transport interpolation between them. 
\end{itemize}
These properties of the $p$-Wasserstein barycenter are illustrated in Figures \ref{Fig:examples} and \ref{Fig:examples2}. 

\begin{figure}[H]
      \includegraphics[width=0.32\textwidth]{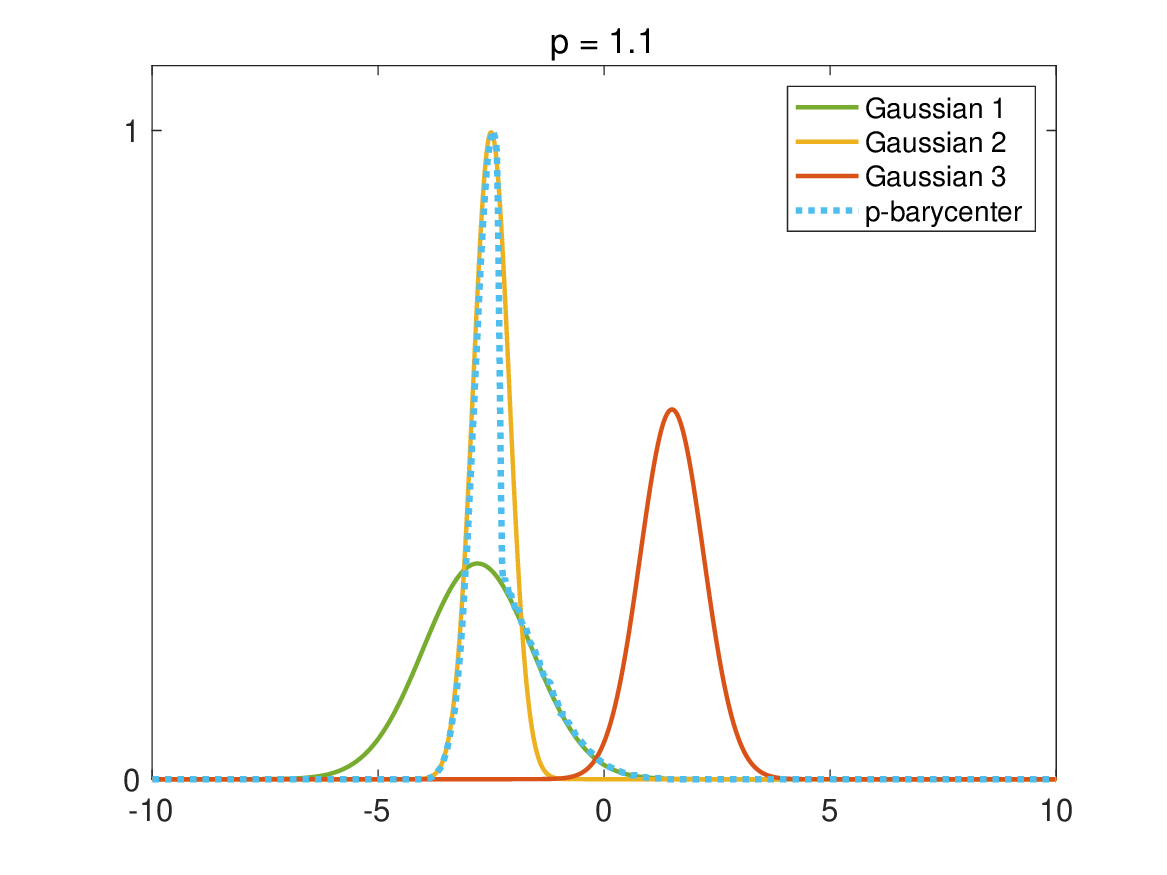}
    \includegraphics[width=0.32\textwidth]{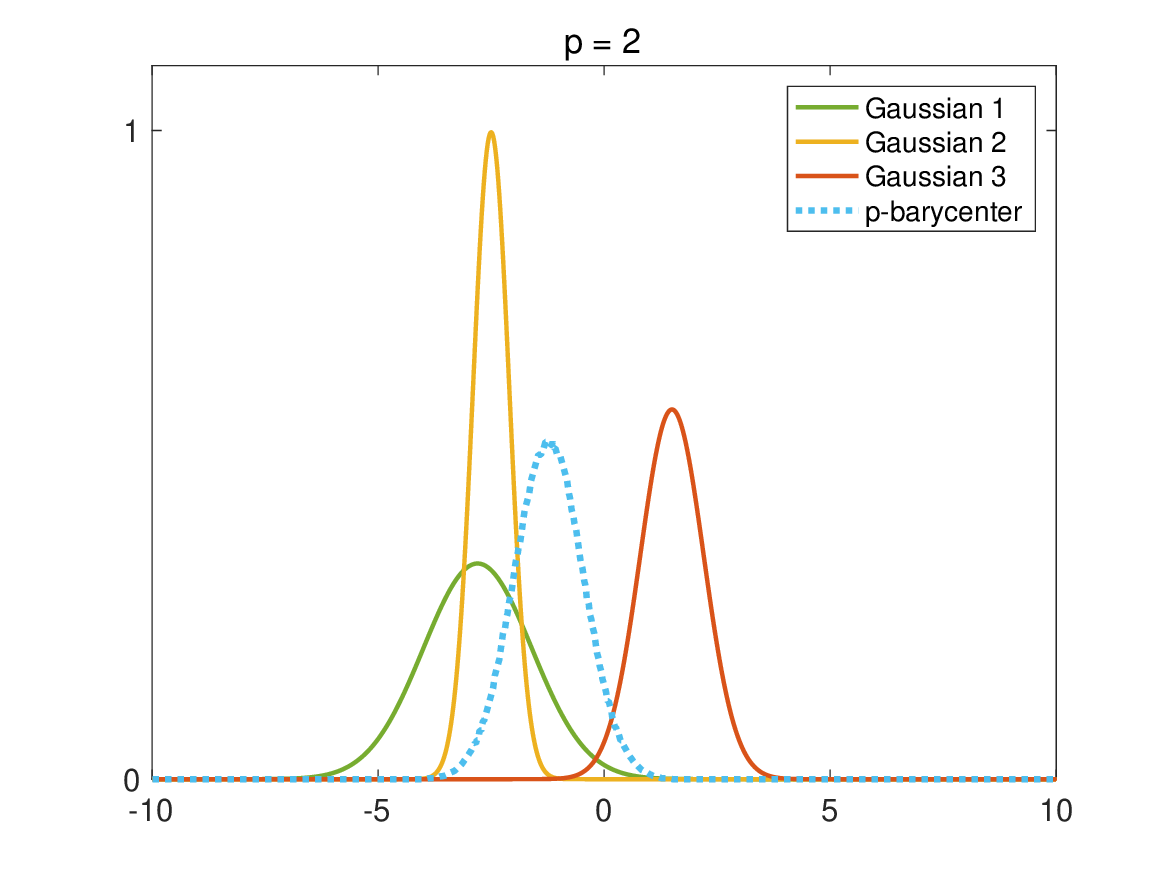}
    \includegraphics[width=0.32\textwidth]{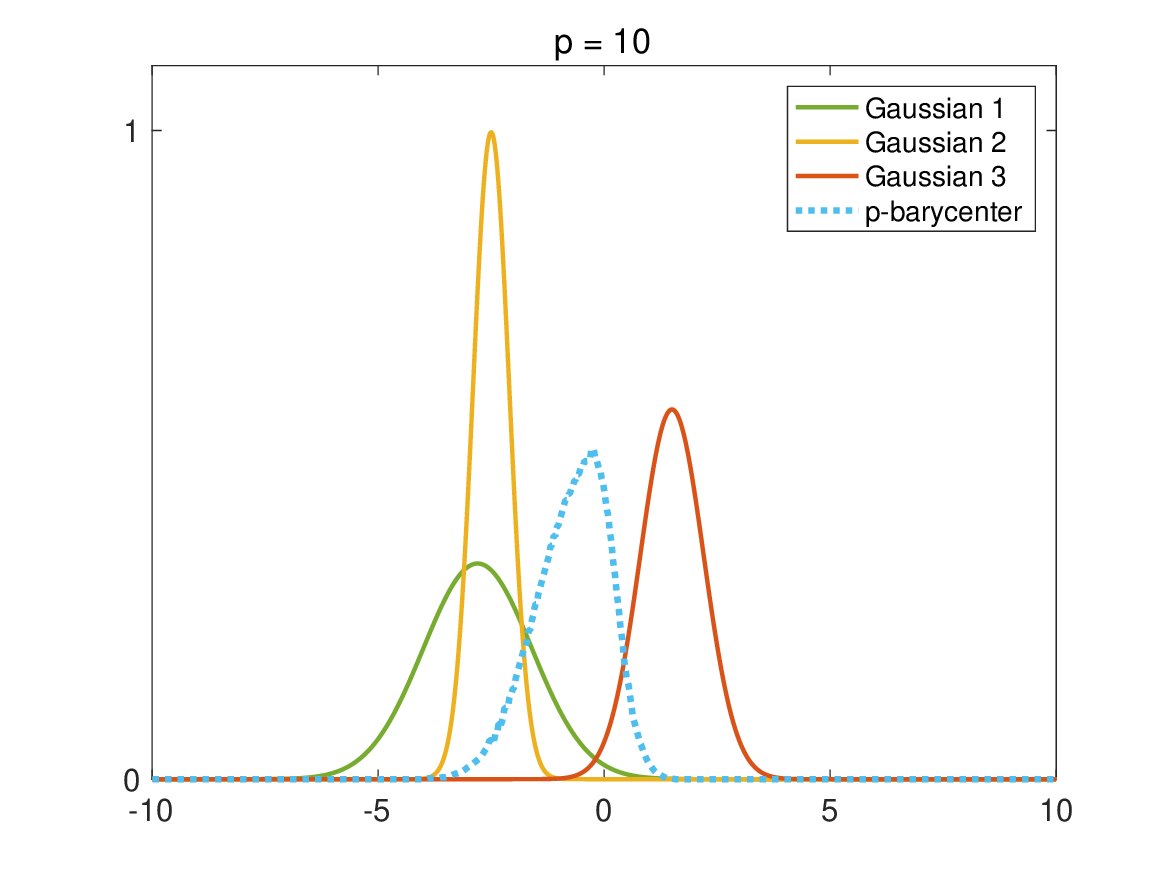} \\
    \includegraphics[width=0.32\textwidth]{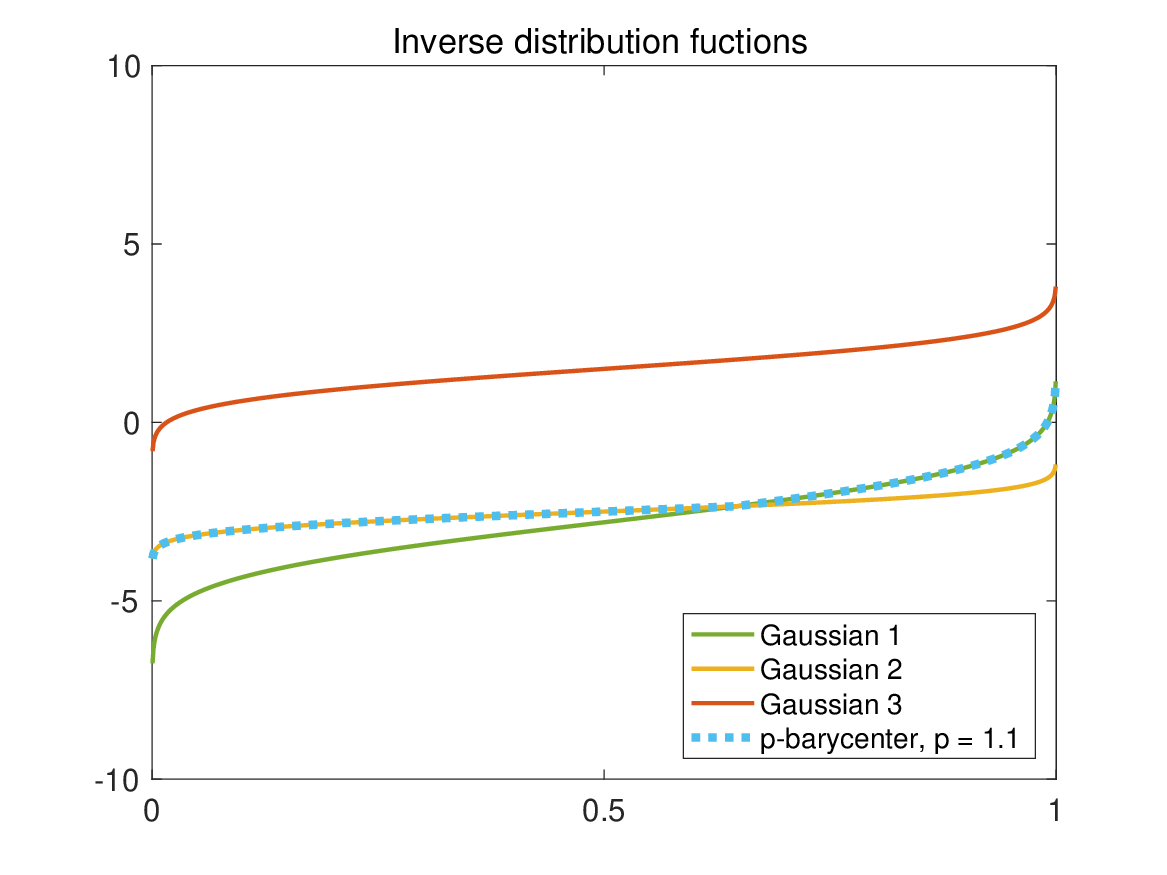}
    \includegraphics[width=0.32\textwidth]{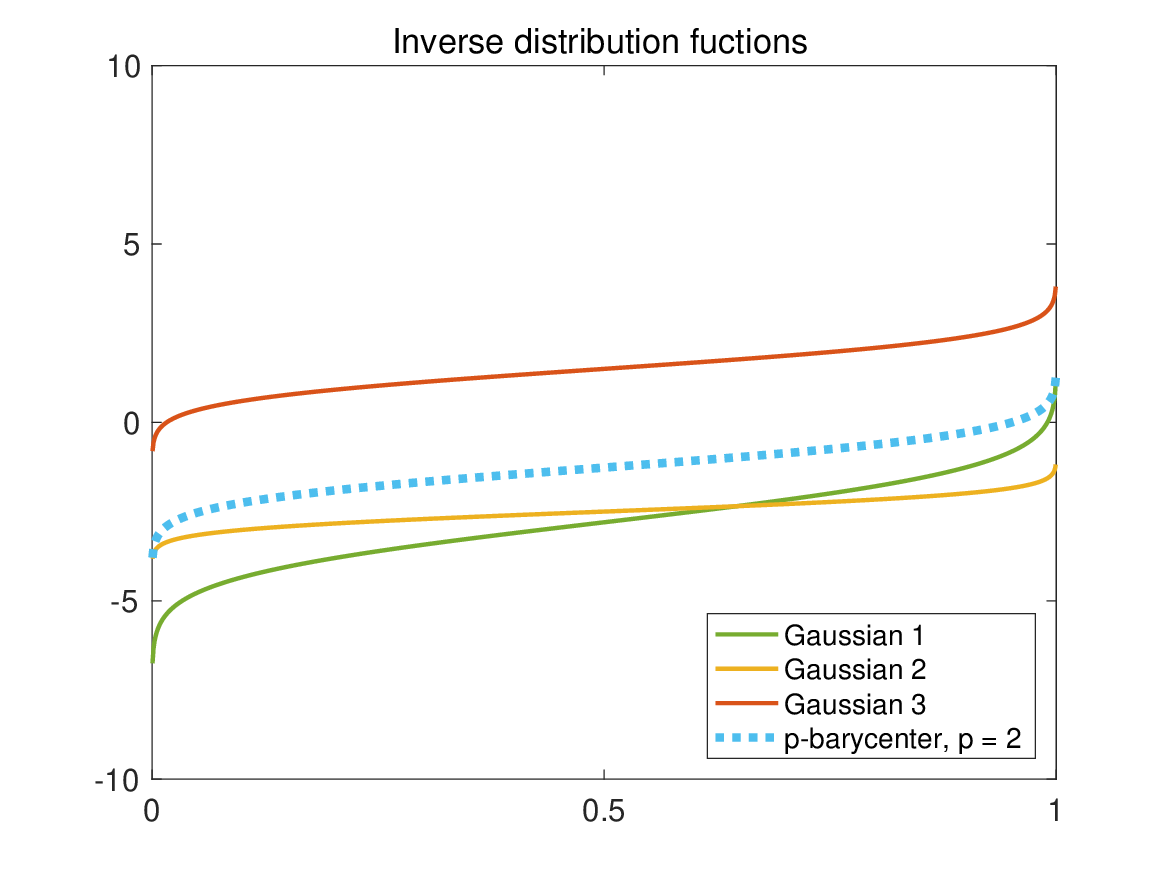}
    \includegraphics[width=0.32\textwidth]{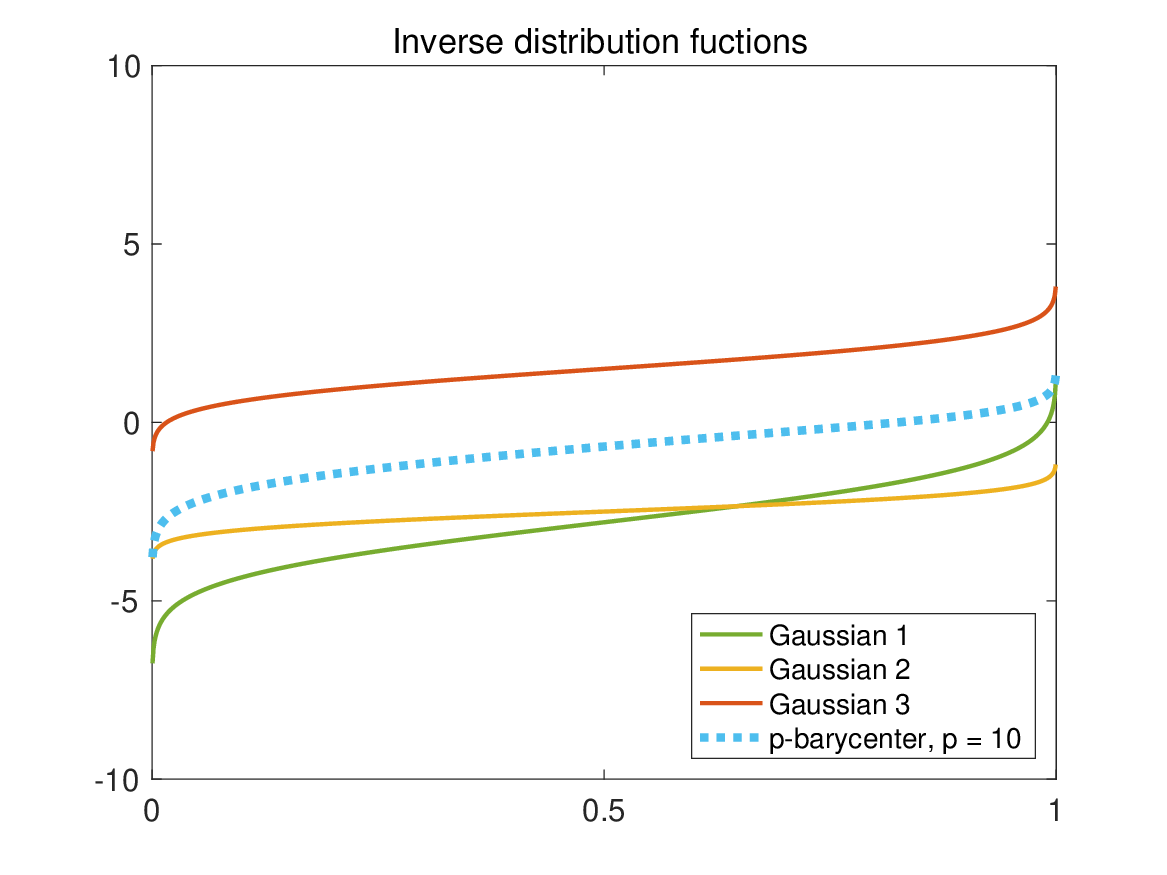}
    \caption{\footnotesize 
    $p$-Wasserstein barycenter of three general Gaussians, for different values of $p$. In terms of inverse distribution functions, for $p=1.1$ the barycenter closely corresponds to taking the median, whereas for $p=10$ it is close to the arithmetic mean of the outer inverse distribution functions, confirming the asymptotic formulae \eqref{1Deq1} and \eqref{1Deqinfty}.}
    \label{Fig:examples2}
\end{figure}
In summary, when $p$ is increased, the $p$-barycenter crosses over from representing `typical' data (and ignoring outliers) to representing a transport average between outliers (and ignoring `typical' data).

{
\addtocontents{toc}{\setcounter{tocdepth}{-10}}
\section*{Acknowledgments}
CB was funded by \emph{Deutsche Forschungsgemeinschaft (DFG -- German Research Foundation) -- Project-ID 195170736 -- TRR109}.  
This work was partially carried out while TR was a visiting research fellow at TU München, funded by 
\emph{Deutsche Forschungsgemeinschaft (DFG -- German Research Foundation) -- Project-ID 195170736 -- TRR109}.

CB would like to thank the Max Planck Institute for Mathematics in the Sciences in Leipzig, and TR would like to thank the Department of Mathematics at TU Munich for their kind hospitality. 

\addtocontents{toc}{\setcounter{tocdepth}{1}}
}

\bigskip

\begin{thebibliography}{MMMM99}

\bibitem[AC11]{AC11}
	\textsc{M. Agueh}, and \textsc{G. Carlier},
	\newblock Barycenters in the Wasserstein space.
	\newblock \doi{10.1137/100805741}{\emph{SIAM Journal on Mathematical Analysis}} \textbf{43} (2011), no. 2, 904--924.
	\newblock \mr{2801182}.
	\newblock \zbl{0889.49030}.
	\hfill


\bibitem[Bre91]{Bre91}
	\textsc{Y. Brenier},
	\newblock Polar factorization and monotone rearrangement of vector-valued functions.
	\newblock \doi{10.1002/cpa.3160440402}{\emph{Communications on Pure and Applied Mathematics}} \textbf{44} (1991), no. 4, 375--483.
	\newblock \mr{1100809}.
	\newblock \zbl{0738.46011}.
	\hfill

\bibitem[BFR24]{BFR24-h}
    \textsc{C. Brizzi, G. Friesecke}, and \textsc{T. Ried},
    \newblock $h$-Wasserstein barycenters.
    \newblock \emph{Preprint} (2024),
    \newblock \arxiv{2402.13176}.
    \hfill

\bibitem[BBI01]{BBI01}
    \textsc{D. Burago, Y. Burago}, and \textsc{S. Ivanov},
    \newblock \doi{10.1090/gsm/033}{\emph{A Course in Metric Geometry}}.
    \newblock Graduate Studies in Mathematics Vol.\ 33.
    \newblock American Mathematical Society, Providence, RI, 2001.
    \newblock \mr{1835418}.
    \newblock \zbl{0981.51016}.
    \hfill

\bibitem[Car03]{Car03} 
    \textsc{G. Carlier}, 
    \newblock On a class of multidimensional optimal transportation problems.
    \newblock \doi{}{\emph{Journal of Convex Analysis}} \textbf{10} (2003), no. 2, 517--529.
    \newblock \mr{2044434}.
    \newblock \zbl{1084.49037}.
    \hfill

\bibitem[CE10]{CE10}
	\textsc{G. Carlier}, and \textsc{I. Ekeland},
	\newblock Matching for teams.
	\newblock \doi{10.1007/s00199-008-0415-z}{\emph{Economic Theory}} \textbf{42} (2010), no. 2, 397--418.
	\newblock \mr{2564442}.
	\newblock \zbl{1183.91112}.
	\hfill

\bibitem[CP21]{CP21}
    \textsc{N.-P. Chung}, and \textsc{M.-N. Phung},
    \newblock Barycenters in the Hellinger-Kantorovich Space.
    \newblock \doi{10.1007/s00245-020-09695-y}{\emph{Applied Mathematics and Optimization}} \textbf{84} (2021), no. 2, 1791--1820.
    \newblock \mr{4304889}.
    \newblock \zbl{1480.49014}.
    \hfill
    
\bibitem[DP15]{DP15}
    \textsc{L. De Pascale},
    \newblock Optimal transport with {C}oulomb cost. {A}pproximation and duality.
    \newblock \doi{10.1051/m2an/2015035}{\emph{ESAIM: Mathematical Modelling and Numerical Analysis}} \textbf{49} (2015), no. 6, 1643--1657.
    \newblock \mr{3423269}.
    \newblock \zbl{1330.49048}.
    \hfill

\bibitem[EG15]{EG15}
    \textsc{L. C. Evans, R. F. Gariepy},
    \newblock {{Measure theory and fine properties of functions}}.
    \newblock Textbooks in Mathematics, CRC Press, Boca Raton, FL,(revised), 2015.
    \newblock \mr{3409135}.
    \newblock \zbl{}.
    \hfill

    
\bibitem[Fed96]{Fed96}
    \textsc{H. Federer},
    \newblock \doi{10.1007/978-3-642-62010-2}{\emph{Geometric Measure Theory}}.
    \newblock Reprint of the 1969 edition. 
    \newblock Classics in Mathematics, Springer Berlin, Heidelberg, 1996.
    \newblock \mr{0257325}.
    \newblock \zbl{0874.49001}.
    \hfill

\bibitem[Fri24]{F24}
	\textsc{G. Friesecke}, 
	\newblock \doi{10.1137/1.9781611978094}{\emph{Optimal Transport: a comprehensive introduction to modeling, analysis, simulation, applications}}.
	\newblock Society for Industrial and Applied Mathematics, Philadelphia, PA, 2024.
	\hfill

	
\bibitem[FMS21]{FMS21}
	\textsc{G. Friesecke, D. Matthes}, and \textsc{B. Schmitzer},
	\newblock Barycenters for the Hellinger-Kantorovich distance over $\RR^d$.
	\newblock \doi{10.1137/20M1315555}{\emph{SIAM Journal on Mathematical Analysis}} \textbf{53} (2021), no. 1, 62--110.
	\newblock \mr{4194317}.
	\newblock \zbl{1457.49032}.
	\hfill

\bibitem[GMC96]{GMC96}
	\textsc{W. Gangbo},  and \textsc{R.J. McCann},
	\newblock The geometry of optimal transportation.
	\newblock \doi{10.1007/BF02392620}{\emph{Acta Mathematica}} \textbf{177} (1996), no. 2, 113--161.
	\newblock \mr{1440931}.
	\newblock \zbl{0887.49017}.
	\hfill

\bibitem[GS98]{GS98}
	\textsc{W. Gangbo}, and \textsc{A. \'Swie\k{}ch},
	\newblock Optimal maps for the multidimensional Monge-Kantorovich problem.
	\newblock \doi{10.1002/(SICI)1097-0312(199801)51:1<23::AID-CPA2>3.0.CO;2-H}{\emph{Communications on Pure and Applied Mathematics}} \textbf{51} (1998), no. 1, 23--45.
	\newblock \mr{1486630}.
	\newblock \zbl{0889.49030}.
	\hfill

\bibitem[Hei02]{Hei02}
    \textsc{H. Heinich},
    \newblock Problème de Monge pour $n$ probabilités.
    \newblock \doi{10.1016/S1631-073X(02)02341-5}{\emph{Comptes Rendus Mathématique. Académie des Sciences, Paris}} \textbf{334} (2002), no. 9, 793--795.
    \newblock \mr{1905042}.
    \newblock \zbl{0996.60025}.
    \hfill


\bibitem[Kel84]{K84}
    \textsc{H.G. Kellerer},
    \newblock Duality theorems for marginal problems.
    \newblock \doi{10.1007/BF00532047}{\emph{Zeitschrift für Wahrscheinlichkeitstheorie und Verwandte Gebiete}} \textbf{67} (1984), no. 4, 399--432.
    \newblock \mr{0761565}.
    \newblock \zbl{0535.60002}.
    \hfill

\bibitem[KP14]{KP14}
	\textsc{Y.-H. Kim}, and \textsc{B. Pass},
	\newblock A general condition for Monge solutions in the multi-marginal optimal transport problem.
	\newblock \doi{10.1137/130930443}{\emph{SIAM Journal on Mathematical Analysis}} \textbf{46} (2014), no. 2, 1538--1550.
	\newblock \mr{3190751}.
	\newblock \zbl{1293.49109}.
	\hfill

\bibitem[KP17]{KP17}
	\textsc{Y.-H. Kim}, and \textsc{B. Pass},
	\newblock Wasserstein barycenters over Riemannian manifolds.
	\newblock \doi{10.1016/j.aim.2016.11.026}{\emph{Advances in Mathematics}} \textbf{307} (2017), 640--683.
	\newblock \mr{3590527}.
	\newblock \zbl{1373.60006}.
	\hfill

\bibitem[LL17]{LL17}
	\textsc{T. Le Gouic}, and \textsc{J.-M. Loubes},
	\newblock Existence and consistency of Wasserstein barycenters.
	\newblock \doi{10.1007/s00440-016-0727-z}{\emph{Probability Theory and Related Fields}} \textbf{168} (2017), no. 3-4, 901--917.
	\newblock \mr{3663634}.
	\newblock \zbl{1406.60019}.
	\hfill

\bibitem[Pas11]{Pas11}
	\textsc{B. Pass},
	\newblock Uniqueness and Monge solutions in the multimarginal optimal transportation problem.
	\newblock \doi{10.1137/100804917}{\emph{SIAM Journal on Mathematical Analysis}} \textbf{43} (2011), no. 6, 2445--2777.
	\newblock \mr{2873240}.
	\newblock \zbl{1248.49061}.
	\hfill	

\bibitem[Pas12]{Pas12}
	\textsc{B. Pass},
	\newblock On the local structure of optimal measures in the multi-marginal optimal transportation problem.
	\newblock \doi{10.1007/s00526-011-0421-z}{\emph{Calculus of Variations and Partial Differential Equations}} \textbf{43} (2012), no. 3-4, 529--536.
	\newblock \mr{2875651}.
	\newblock \zbl{1231.49037}.
	\hfill	

\bibitem[Pas14]{Pas14}
	\textsc{B. Pass},
	\newblock Multi-marginal optimal transport and multi-agent matching problems: uniqueness and structure of solutions.
	\newblock \doi{10.3934/dcds.2014.34.1623}{\emph{Discrete and Continuous Dynamical Systems. Series A}} \textbf{34} (2014), no. 4, 1623--1639.
	\newblock \mr{3121634}.
	\newblock \zbl{1278.49054}.
	\hfill	


\bibitem[San15]{San15}
	\textsc{F. Santambrogio}, 
	\newblock \doi{10.1007/978-3-319-20828-2}{\emph{Optimal transport for applied mathematicians}}.
	\newblock Calculus of Variations, PDEs, and Modeling. 
	\newblock Progress in Nonlinear Differential Equations and their Applications \textbf{87}, 
	\newblock Birkhäuser, Cham, 2015.
	\newblock \mr{3409718}.
	\newblock \zbl{1401.49002}.
	\hfill

\bibitem[Vil09]{Vil09}
	\textsc{C. Villani},
	\newblock \doi{10.1007/978-3-540-71050-9}{\emph{Optimal transport: old and new}}. 
	\newblock Grundlehren der Mathematischen Wissenschaften \textbf{338}, 
	\newblock Springer-Verlag, Berlin, 2009.
	\newblock \mr{2459454}.
	\newblock \zbl{1156.53003}.
	\hfill

\end{thebibliography}
\end{document}